\documentclass{article}

\usepackage[english]{babel}

\usepackage[letterpaper,top=2cm,bottom=2cm,left=3cm,right=3cm,marginparwidth=1.75cm]{geometry}

\usepackage{amsmath,amssymb,amsthm}
\usepackage{mathtools}
\usepackage{graphicx}
\usepackage[colorlinks=true, allcolors=blue]{hyperref}

\usepackage{algorithm} 
\usepackage{algorithmic}  
\usepackage[algo2e]{algorithm2e} 

\newtheorem{thm}{Theorem}[section]
\newtheorem{theorem}{Theorem}[section]

\newtheorem{proposition}[thm]{Proposition}

\newtheorem{lemma}[thm]{Lemma}

\theoremstyle{remark}
\newtheorem{remark}{Remark}

\def\cA{{\cal A}}

\def\cF{{\cal F}}

\def\cH{{\cal H}}
\def\cI{{\cal I}}

\def\cL{{\cal L}}

\def\R{{\mathbb{R}}}

\DeclareMathOperator{\spann}{span}
\DeclareMathOperator*{\argmin}{arg\,min}

\title{Subgame Perfect Methods in Nonsmooth Convex Optimization}
\author{
    Benjamin Grimmer\footnote{Johns Hopkins University, Department of Applied Mathematics and Statistics, \texttt{grimmer@jhu.edu}}
    \and
    Alex L.\ Wang\footnote{Purdue University, Daniels School of Business, \texttt{wang5984@purdue.edu}}
}
\date{}

\begin{document}
\maketitle

\begin{abstract}
    This paper considers nonsmooth convex optimization with either a subgradient or proximal operator oracle. In both settings, we identify algorithms that achieve the recently introduced game-theoretic optimality notion for algorithms known as subgame perfection.
    Subgame perfect algorithms meet a more stringent requirement than just minimax optimality. Not only must they provide optimal uniform guarantees on the entire problem class, but also on any subclass determined by information revealed during the execution of the algorithm. 
    In the setting of nonsmooth convex optimization with a subgradient oracle, we show that the Kelley cutting-plane-Like Method due to~\cite{drori2016Kelley} is subgame perfect.
    For nonsmooth convex optimization with a proximal operator oracle, we develop a new algorithm, the Subgame Perfect Proximal Point Algorithm (SPPPA), and establish that it is subgame perfect. 
    Both of these methods solve a history-aware second-order cone program within each iteration, independent of the ambient problem dimension, to plan their next steps. This yields performance guarantees that are never worse than the minimax optimal guarantees and often substantially better. Preliminary numerical experiments with these methods, especially SPPPA, show strong performance with limited memory and inexact implementations.
\end{abstract}

\section{Introduction}

Consider a (potentially nonsmooth) convex minimization problem 
$$ \min_{x\in\mathbb{R}^d} f(x)$$
with unknown minimizer $x_\star$. We are particularly interested in the high-dimensional regime where $d$ may be arbitrarily large. To design iterative methods for such problems, one typically assumes an oracle access model for $f$. The two most fundamental oracle models for nonsmooth optimization are the subgradient oracle
$$x\mapsto (f(x),g)\quad \text{where} \quad g\in \partial f(x) \coloneqq \{g\in\mathbb{R}^d \mid f(z) \geq f(x) + \langle g, z-x\rangle \quad \forall z\in\mathbb{R}^d\}$$
and $g\in\partial f(x)$ may be chosen adversarially; and the proximal operator oracle, 
$$ x\mapsto (f(y),y)\quad\text{where}\quad y=\mathrm{prox}_{L,f}(x)\coloneqq \argmin_{y\in\mathbb{R}^d} f(y) + \frac{L}{2}\|y-x\|^2 $$
for a given proximal parameter $L>0$. Throughout, we consider algorithms 
given a fixed budget of $N$ evaluations of one of these two oracles to produce approximate minimizers.

These two models of optimization have seen substantial study. The study of subgradient methods (those utilizing a subgradient oracle) was pioneered early on by~\cite{Shor1985,Lemarechal1975,Wolfe1975}.
Similarly, proximal methods (those utilizing a proximal operator oracle) have a rich history of study dating back to~\cite{martinet1970regularisation,martinet1972determination,Rockafellar1976_monotone,brezis1978produits}.

A classical goal in algorithm design is minimax optimal performance on a given problem class. That is, one often wants an algorithm with the best worst-case performance. To be concrete, consider designing a subgradient method, with a budget of $N$ subgradient oracle queries, for minimizing an $M$-Lipschitz continuous convex function $f$ given an initialization $x_0$ with $\|x_0-x_\star\|\leq R$ for some minimizer $x_\star$ of $f$. Denote the set of all such instances $(f,x_0)$ by $\mathcal{F}_{M,R}$. Denote by $\mathcal{A}$ the set of all (deterministic) subgradient-span methods with iterates satisfying
\begin{equation}\label{eq:subgrad-span}
    x_{n} \in x_0 + \mathrm{span}\{ g_0,\dots, g_{n-1}\}, \qquad g_i \in\partial f(x_i)
\end{equation}
for $n=1,\dots, N$.
That is, each $A\in\mathcal{A}$ is a deterministic map from a history (possibly empty) of query-response pairs $\left\{(x_0,f_0,g_0),\dots, (x_{n-1}, f_{n-1}, g_{n-1})\right\}$ to the next
query point $x_n\in x_0+\spann\{g_0,\dots, g_{n-1}\}$. Here, $f_{i}$ is shorthand for $f(x_i)$.
When $A\in\mathcal{A}$ and $(f,x_0)\in\mathcal{F}_{M,R}$ are clear from context, 
we will let $x_0, x_1,\dots,x_N$ denote the iterates produced by $A$ on the instance $(f,x_0)$. The worst-case final objective gap for $A$ is given by
$$ \max_{(f,x_0)\in \mathcal{F}_{M,R}} f(x_N) - f(x_\star). $$
An algorithm is ``minimax optimal'' if it attains the optimal worst-case performance guarantee:
\begin{equation}
    \min_{A\in \mathcal{A}}\max_{(f,x_0)\in \mathcal{F}_{M,R}} f(x_N) - f(x_\star).
    \label{eq:apriori-minimax-optimal}
\end{equation}

Methods that are within a universal constant factor of being minimax optimal for a wide range of problem/algorithm settings were designed in the 1980s by the seminal works of~\cite{Nesterov1983} and~\cite{nemirovskij1983problem}. Since then,
exactly minimax optimal algorithms have been designed for many settings (see, for example, the works of~\cite{Kim2016optimal,VanScoy2018,Cyrus2018,Park2021Factor,taylor2022optimal}), supported by matching lower bounding instances of~\cite{drori2017LowerBoundOGM} and~\cite{drori2022LowerBoundITEM}. In large part, these developments were spurred by the Performance Estimation Program (PEP) methodology of~\cite{drori2012PerformanceOF,taylor2017interpolation,taylor2017CompositePEP}.

Alas, minimax optimal algorithms often do not offer the best performance on ``typical'' or ``real-world'' problems.
In essence, minimax optimality is a uniform requirement and does not preclude an algorithm from performing ``suboptimally'' when $(f,x_0)\in\cF_{M,R}$ is not chosen adversarially.

Recently, a theoretically grounded approach to strengthening minimax optimality was proposed by the present authors in~\cite{SPGM}. Instead of only requiring optimal guarantees on the \emph{entire} problem class, we require that the algorithm, \emph{at every iteration $n$}, offers the optimal guarantee against the \emph{subclass} of remaining problem instances agreeing with the oracle responses so far. That is, let $\cH=\{(x_i,f_i,g_i)\}_{i=0}^{n-1}$ denote the history of query-response pairs seen up to iteration $n$. Denote by $\mathcal{F}_{M,R}^\cH\subseteq \mathcal{F}_{M,R}$ the set of remaining problem instances $(f,x_0)$ agreeing with the given history, i.e., $f_i=f(x_i)$ and $g_i\in\partial f(x_i)$. Similarly, denote by $\mathcal{A}^\cH$ the set of all subgradient methods providing a continuation of the first $n$ iterates $x_0,\dots, x_{n-1}$ for the remaining $N-n+1$ iterations.
% \footnote{
Formally, $\cA^\cH$ is the subset of $A\in\cA$ mapping
$\{(x_i, f_i, g_i)\}_{i=0}^{j-1}\mapsto x_j$ for all $j = 1,\dots, n-1$.
% }
Then, an algorithm is ``subgame perfect'' if at every iteration $n$ and for any observed history $\cH$ up to iteration $n$, it attains the best possible worst-case final objective gap in the remaining iterations
\begin{equation}
    \min_{A\in \mathcal{A}^\cH}\max_{(f,x_0)\in \mathcal{F}_{M,R}^\cH} f(x_N) - f(x_\star). \label{eq:dynamic-minimax-optimal}
\end{equation}
This is a strengthening of minimax optimality, since at iteration $n=0$, \eqref{eq:dynamic-minimax-optimal} reduces to~\eqref{eq:apriori-minimax-optimal}. The terminology ``subgame perfect'' stems from a game-theoretic perspective on~\eqref{eq:apriori-minimax-optimal} and~\eqref{eq:dynamic-minimax-optimal}: A minimax optimal method and a worst-case adversarial selection of the problem instance $(f,x_0)$ together provide a saddle point, or Nash Equilibrium, of~\eqref{eq:apriori-minimax-optimal}. 
In other words, a minimax optimal method corresponds to an optimal algorithm for playing \eqref{eq:apriori-minimax-optimal} \emph{if} $(f,x_0)$ is chosen adversarially.
On the other hand, a subgame perfect method and an associated subgame perfect adversary form a Subgame Perfect Nash Equilibrium for the associated sequential optimization game.
Intuitively, a subgame perfect method corresponds to an algorithm that optimally capitalizes on imperfect play by the adversary, while maintaining minimax optimality against fully adversarial play. This offers a principled beyond-worst-case guarantee for convex optimization.
We refer readers to~\cite{SPGM} for a more in-depth discussion of this game-theoretic perspective. Additionally, numerical experiments are given in~\cite{SPGM} and~\cite{ASPGM}, which show strong practical gains from such refinements. 

This subgame perfect criterion was first proposed in the context of smooth convex minimization. Here we show that subgame perfect methods can also be designed for nonsmooth convex minimization. First, we show that the Kelley cutting plane-Like Method (KLM) of~\cite{drori2016Kelley} is in fact subgame perfect for Lipschitz convex minimization with a subgradient oracle. Second, we design a new 
subgame perfect algorithm for convex optimization with a proximal oracle, which we call the 
Subgame Perfect Proximal Point Algorithm (SPPPA).

\paragraph{Contributions and organization.} In Section~\ref{sec:Kelley}, we begin by describing a minimax optimal subgradient method, KLM, of~\cite{drori2016Kelley} and state its dynamic guarantees. We then show that this method is subgame perfect by constructing a matching dynamic lower bound. In Section~\ref{sec:SPPPA}, we turn to designing a new subgame perfect proximal point-type method. Our method, SPPPA, generalizes the known minimax-optimal proximal point method, improving it with a dynamic reoptimization of its inductive hypothesis at every iteration. In Section~\ref{sec:SPPPA-proof}, we prove that SPPPA is subgame perfect by examining associated dual certificates generated by these reoptimizations at each step to construct matching dynamic lower bounds.
Finally, in Section~\ref{sec:pragmatics}, we discuss practical implementations of KLM and SPPPA and present preliminary numerical experiments.

For both algorithms KLM and SPPPA, we will invoke a ``high-dimensional'' assumption $d\geq N+1$ when proving our dynamic lower bounds. On the other hand, the upper bounds (i.e., the convergence guarantees) hold even in the setting $N\gg d$.
\section{Subgame Perfect Subgradient Method}\label{sec:Kelley}

% Throughout, we will consider optimization of closed convex proper functions $f:\R^d\to\R \cup\{\infty\}$ and will use an inner product $\langle\cdot,\cdot\rangle$ and associated norm $\|\cdot\|$ in $\mathbb{R}^d$.
In this section, we consider minimization of an unconstrained convex $M$-Lipschitz function $\min_{x\in\R^d} f(x)$.
That is, $f:\R^d\to\R$ is convex and $M$-Lipschitz.
Note that this ensures $\|g\|\le M$ for every $g\in\partial f(x)$ and $x\in\R^d$. 
We assume a minimizer $x_\star\in\arg\min f$ exists and that the given initialization $x_0$ satisfies $\|x_0-x_\star\|\le R$. All norms are the Euclidean norm associated with the given inner product $\langle \cdot,\cdot\rangle$.

Given the iterates $x_0,\dots, x_N$, it is convenient to reason about the associated first-order (subgradient) data
$\{(x_i,f_i,g_i)\}_{i\in\cI}$ with $f_i=f(x_i)$ and $g_i\in\partial f(x_i)$ where $\cI=\{0,\dots,N\}$. At times, it will be useful to include the minimizer $x_\star$ in our set of observations; to do so, we set $g_\star=0$, $f_\star=f(x_\star)$ and $\cI_\star =\{0,\dots,N,\star\}$. Convexity of $f$ ensures that
\[
    Q_{i , j} \coloneqq  f_i - f_j - \langle g_j, x_i-x_j \rangle \geq 0\qquad\forall i,j\in \cI_\star,
\]
and the $M$-Lipschitz continuity of $f$ ensures that $S_i \coloneqq  M^2-\|g_i\|^2 \geq 0$ for all $i\in \cI_\star$.
That is, the nonnegativity of the above quantities is a necessary condition on the first-order data to have come from some $M$-Lipschitz convex function. A special case of the interpolation theorem of~\cite[Theorem 3.5]{taylor2017CompositePEP} states that this condition is also sufficient:
\begin{lemma}[{\cite[Theorem 3.5]{taylor2017CompositePEP}}]
Let $\{(x_i, f_i,g_i)\}_{i\in\cI_\star}\subseteq \R^d\times\R\times\R^d$. The quantities $Q_{i,j}$ and $S_i$ are nonnegative for all $i,j\in\cI_\star$ if and only if there exists an $M$-Lipschitz convex function $f$ satisfying
$$f_i = f(x_i)\quad\text{and}\quad g_i\in\partial f(x_i)\qquad\forall i\in\cI_\star.$$
\end{lemma}
This lemma can also be proved directly by setting $f(y) = \max_{i\in\cI_\star}\{f_i + \langle g_i, y-x_i \rangle \}$.
% \begin{theorem}
%     For any set of observations $\{(x_i,f_i,g_i)\}_{i\in\cI}$, there exists an $M$-Lipschitz, convex $f$ with $f(x_i)=f_i$ and $g_i\in\partial f(x_i)$ for every $i\in\cI$ if and only if $Q_{i,j}\geq 0$ and $S_i\geq 0$ for all $i,j\in\cI$. 
% \end{theorem}

\subsection{The Kelley-Like Method of~\cite{drori2016Kelley}}
The Kelley cutting plane Method of \cite{Kelley1960} is an early example of a ``bundle method'': an algorithm maintaining a bundle of first-order information $\{(x_i,f_i,g_i)\}_{i=0}^{n-1}$ at each iteration $n$, used to inform the selection of the next iterate $x_n$.
Such methods have a long history of practical success, being studied theoretically by~\cite{Lemarechal1975,Wolfe1975,Kiwiel1983,Hiriart1993,Ruszczynski2006,Warren2010,Lan2015,Du2017,DiazGrimmer2023,Liang2021} among many others.

\cite{drori2016Kelley} designed a variant of the Kelley cutting plane method seeking to optimize the iterate selection against the hardest possible function consistent with the observed subgradients. Their resulting method was developed by reformulating and relaxing the search for the hard instances through a series of mathematical programs, eventually providing a tractable planning problem to be solved at every iteration.
From this, they propose their Kelley cutting plane-Like Method (KLM), which solves this planning problem at every iteration to dynamically respond to observed subgradients. In particular, consider an observed history of first-order evaluations prior to iteration $n$, $\cH=\{(x_i,f_i,g_i)\}_{i=0}^{n-1}$ where $f_i = f(x_i)$ and $g_i\in\partial f(x_i)$. KLM sets 
$f_{n-1/2}= \min_{i=0,\dots,n-1} f_i$ as the minimum function value seen, then solves the second-order cone program
\begin{equation} \label{eq:Kelley-solve}
    \Theta_n = \begin{cases}
        \max\limits_{y\in\R^d, \zeta, t\in\R} & f_{n-1/2} - t\\
        \mathrm{s.t.} & t \geq f_i + \langle g_i, y-x_i \rangle \qquad \forall i=0,\dots, n-1\\
        & f_{n-1/2} - M\zeta \leq t\\
        &\|y-x_0\|^2 + (N-n+1)\zeta^2 \leq R^2.
    \end{cases}
\end{equation}

Observe that (i) $\Theta_n \geq 0$ since $y=x_\star,\zeta=0,t=f_{n-1/2}$ is a feasible solution and (ii) there exists an optimal solution with $y\in x_0 + \mathrm{span}\{g_0,\dots,g_{n-1}\}$. 
Letting $(y, \zeta, t)$ denote some such optimal solution to this problem in the $n$th iteration, KLM will iterate by setting $x_n = y$.  Since $y\in x_0 +\mathrm{span}\{g_0,\dots,g_{n-1}\}$, this update falls within the (sub)gradient span model~\eqref{eq:subgrad-span}.

The KLM is formalized in Algorithm~\ref{alg:Kelley}.\footnote{In~\cite{drori2016Kelley}, it is noted that one can freely, at some iterations, take a regular subgradient method step instead of the optimized step resulting from the planning subproblem. We omit this from our presentation as the resulting algorithm is no longer subgame perfect.} As input, this method requires bounds on the problem Lipschitz constant $M$, the distance to a minimizer $R$, and the total iteration budget $N$.

\begin{algorithm}[H]
    \caption{The Kelley cutting plane-Like Method (KLM)}
    \label{alg:Kelley}
    Given convex $f$, initial iterate $x_0$, Lipschitz constant $M$, distance bound $R$, iteration count $N$
    \begin{itemize}
        \item For $n=0$, define $\Theta_0 = MR/\sqrt{N+1}$
        \item For $n = 1,2,\dots,N$
        \begin{itemize}
            \item Query oracle to get $f_{n-1}=f(x_{n-1})$ and $g_{n-1} \in\partial f(x_{n-1}) $.
            \item Let $f_{n-1/2} = \min_{i\in 0,\dots, n-1}f_i$ and solve~\eqref{eq:Kelley-solve}, generating $\Theta_n$ and optimal $(y,\zeta,t)$.
            \item Set $x_n = y$.
        \end{itemize}
        \item \textbf{Output:} $\bar{x}_N=(1-\beta_N)x_m+\beta_N x_N$, where $m\in\arg\min_{0\le i<N}f_i$ and $\beta_N$ is an optimal dual multiplier for $f_{N-1/2}-M\zeta\le t$ in the final subproblem~\eqref{eq:Kelley-solve}.
    \end{itemize}
    \end{algorithm}

\begin{remark}
As stated, KLM requires knowledge (or estimates) of multiple parameters as well as a full memory of past subgradients. This makes KLM difficult to implement in practice. Our numerics in Section~\ref{sec:pragmatics} will consider both full-memory and preliminary limited-memory variants of the KLM scheme.
As we will see, full-memory KLM is highly performant, whereas limited-memory KLM suffers whenever important ``facets'' leave the ``bundle'' of past memory.
Important future directions include removing the dependence of KLM on the unknown parameters, designing performant limited-memory variants, and addressing the constrained setting.
\end{remark}

The design philosophy of Drori and Teboulle is very similar in nature to the motivation behind the subgame perfect framework of~\cite{SPGM}, making KLM a plausible candidate for a subgame perfect subgradient method. In~\cite{drori2016Kelley}, they proved the following convergence upper bound for KLM.
\begin{theorem} \label{thm:Kelley-rate}
    For any $(f,x_0)\in\mathcal{F}_{M,R}$ and $N\geq 0$, the output $\bar{x}_N$ of KLM is guaranteed to satisfy
    $$ f(\bar{x}_N) - f(x_\star) \leq  \Theta_N \leq \dots \leq  \Theta_1 \leq \Theta_0= \frac{MR}{\sqrt{N+1}}.$$
\end{theorem}
\noindent The derivation of this dynamic sequence of upper bounds is nontrivial, requiring substantial intermediate developments. We refer readers to~\cite{drori2016Kelley} for its proof.

In addition to this upper bound, Drori and Teboulle provide a lower bounding instance (assuming $d\geq N+1$) showing that no subgradient method can achieve a worst-case convergence guarantee strictly better than $MR/\sqrt{N+1}$. As a result, KLM is minimax optimal for Lipschitz convex minimization in the sense of~\eqref{eq:apriori-minimax-optimal}. However, this single hard instance is insufficient to prove that KLM is subgame perfect. Instead, one needs to construct a dynamic lower bound consistent with any observed history up to iteration $n$ matching the upper bound $\Theta_n$ at iteration $n$. The remainder of this section provides such a construction leading to the following strengthened guarantee for KLM:
\begin{theorem}\label{thm:Kelley-perfection}
    Let $N>0$ and $d\geq N+1$ and consider some iteration $0\leq n\leq N$.
    Suppose $\cH=\{(x_i,f_i,g_i)\}_{i=0}^{n-1}$ is the set of observed first-order history.
    % Assume that $\cH$ is consistent with some $(f,x_0)\in \cF_{M,R}$. Equivalently, assume $\mathcal{F}_{M,R}^\cH$ is nonempty.
    Then,
    % For any $(f,x_0)\in\mathcal{F}_{M,R}$ and $N>0$, consider some iteration $1\leq n \leq N$ and the history $\cH=\{(x_i,f_i,g_i)\}_{i=0}^{n-1}$ of observed first-order information by KLM. Then 
    $$ \min_{A\in \mathcal{A}^\cH}\max_{(f,x_0)\in \mathcal{F}_{M,R}^\cH} f(A(f,x_0)) - f(x_\star) = \Theta_n. $$
    % Hence, KLM's associated dynamic convergence guarantee in Theorem~\ref{thm:Kelley-rate} is as strong as any first-order method can attain. 
    That is, KLM is subgame perfect.
\end{theorem}
When $n=0$, this is equivalent to the hardness result of \cite{drori2016Kelley}. 

\subsection{Proof of Theorem~\ref{thm:Kelley-perfection}}
This section contains a proof of Theorem~\ref{thm:Kelley-perfection}. 
Since the case $n=0$ is covered by \cite{drori2016Kelley}, we will assume $1\leq n \leq N$ and $d\geq N+1$.
Fix a set of observed first-order history $\{(x_i,f_i,g_i)\}_{i=0}^{n-1}$.
Our goal is to construct an instance $(f,x_0)\in\cF_{M,R}^\cH$ so that for all $A\in\cA^\cH$:
\begin{equation}
\label{eq:kelley_hardness}
f(A(f,x_0))- f(x_\star)\geq\Theta_n.
\end{equation}

Since $d\geq N+1$, choose orthonormal vectors $e_n,\dots,e_N$ perpendicular to $\mathrm{span}\{g_0,\dots,g_{n-1}\}$. 
By definition, $\Theta_n$ is the optimal value of \eqref{eq:Kelley-solve}.
% Our goal is to show that for all $A\in\cA^\cH$,
% there exists an instance $(f,x_0)\in \cF_{M,R}^\cH$, satisfying
% \begin{align*}
%     f(A(f,x_0)) - f(x_\star)\geq \Theta_n.
% \end{align*}
% In fact, this instance $(f,x_0)\in \cF_{M,R}^\cH$ can be chosen independently of $A\in\cA^\cH$.
Let $(y, \zeta, t)$ denote an optimizer of \eqref{eq:Kelley-solve}.
% We will take $f$ to be the max-of-affine function corresponding to a set of first-order information.
Define
\begin{equation}
    \label{eq:future_block_defs}
    x_i \coloneqq y-\frac{f_{n-1/2} - t}{M}\sum_{j=n}^{i-1} e_j,\qquad f_i \coloneqq f_{n-1/2}, \qquad g_i\coloneqq M e_i \qquad \text{for all }i=n,\dots,N,
\end{equation}
treating the sum above as $0$ when it is empty and
\[
    x_\star\coloneqq y-\frac{f_{n-1/2} - t}{M}\sum_{j=n}^N e_j, \qquad f_\star\coloneqq t, \qquad g_\star\coloneqq 0.
\]
Combining the given first-order data, $\{(x_i,f_i,g_i)\}_{i=0,\dots,n-1}$, with the constructed first-order data, $\{(x_i,f_i,g_i)\}_{i=n,\dots,N,\star}$, we construct our candidate hard function as the max-of-affine function
\begin{equation}
\label{eq:hard_func_Kelley}
f_{\cH}(x)
\coloneqq \max_{i\in\cI_\star}\{f_i+\langle g_i,x-x_i\rangle\big\}.
\end{equation}

As motivation, the defined points $x_n,\dots,x_N$ happen to be the iterates that KLM would produce if applied to this instance $(f_\cH,x_0)$. Since this fact is not needed for our analysis, we omit its formal proof, which amounts to verifying the needed optimality condition for each successive subproblem.
This justifies the notation $x_n,\dots, x_N, x_\star$ and we refer to these quantities as the ``future iterates.'' 
% This motivates the choice of notation $x_n,\dots,x_N$.
We caution that the $n$ through $N$th iterates produced by some other algorithm $A\in\cA^\cH$ may not coincide with the above construction of $x_n,\dots, x_N$. Nonetheless, it will still hold that the \emph{instance} $(f_\cH,x_0)\in \cF_{M,R}^\cH$ is hard for all $A\in\cA^\cH$ in the sense of \eqref{eq:kelley_hardness}.

Below we verify that $(f_\cH,x_0)\in\cF^\cH_{M,R}$ and prove a needed ``zero-chain'' property. These facts can then be combined to give a short direct proof of Theorem~\ref{thm:Kelley-perfection}.

\begin{proposition}\label{prop:feasibility-Kelley}
    Suppose $d\geq N+1$ and $1\leq n \leq N$. Let $(f_\cH,x_0)$ denote the instance constructed above from $\cH=\{(x_i, f_i, g_i)\}_{i=0}^{n-1}$.
    It holds that $f_\cH$ is convex and $M$-Lipschitz with $f_\cH(x_i)= f_i$ and $g_i \in \partial f_\cH(x_i)$ for all $i\in\cI_\star$
    and $\|x_0-x_\star\|\leq R$.
    In particular, $(f_\cH, x_0)\in \cF_{M,R}^{\cH}$.
\end{proposition}
\begin{proof}
    As $f_\cH$ is a finite maximum of $M$-Lipschitz affine functions, it follows that it is convex and $M$-Lipschitz.
    We next check that $f_\cH(x_i)=f_i$ and $\partial f_\cH(x_i)\ni g_i$ for all $i\in\cI_\star$. Inspecting \eqref{eq:hard_func_Kelley}, we see it suffices to check that the affine function $f_i + \langle g_i, y-x_i\rangle$ is active at $x_i$, i.e., that
    $$Q_{i,j}\coloneqq f_i - f_j - \langle g_j, x_i - x_j\rangle\geq 0\qquad\forall i,j\in\cI_\star.$$

    % Observing $\|g_i\|\leq M$ for all $i\in\cI$, it follows that $f_\cH$ is $M$-Lipschitz.

    We break this verification into cases:
    First, suppose $i\in\{0,\dots,n-1\}$. Then,
    \begin{align*}
        j\in\{0,\dots,n-1\} \ &: \ f_i \geq f_j + \langle g_j, x_i-x_j\rangle,\\
        j\in\{n,\dots,N\} \ &: \ f_i \geq f_{n-1/2} = f_j + \langle g_j, x_i-x_j \rangle, \text{ and}\\
        j = \star \ &: \ f_i \geq f_{n-1/2}\geq f_\star =  f_\star + \langle g_\star, x_i-x_\star \rangle.
    \end{align*}
    The first case comes from the fact that the observed first-order history must come from some convex function; the second case uses the definition of $f_{n-1/2}$ and the orthogonality of $g_j = M e_j$ to the span $\mathrm{span}\{g_0,\dots, g_{n-1}, \dots, g_{j-1}\}$ which contains both $x_i-x_0$ and $x_j -x_0$ for the equality; and the third case uses that 
    $\Theta_n\geq 0$ so that
    $f_\star \leq f_{n-1/2}$ and that $g_\star=0$.
    
    Next, suppose $i\in\{n,\dots,N\}$. Then,
    \begin{align*}
        j\in\{0,\dots,n-1\} \ &: \ f_i = f_{n-1/2} = t + \Theta_n \geq f_j + \langle g_j, y -x_j\rangle + \Theta_n \geq f_j + \langle g_j, x_i -x_j\rangle ,\\
        j\in\{n,\dots,N\} \ &: \ f_i = f_{n-1/2} = f_j \geq f_j + \langle g_j, x_i-x_j \rangle, \text{ and}\\
        j = \star \ &: \ f_i = f_{n-1/2} = f_\star + \Theta_n \geq f_\star = f_\star + \langle g_\star, x_i - x_\star \rangle.
    \end{align*}    
    The first case uses the first constraint in~\eqref{eq:Kelley-solve} and that $\Theta_n \geq 0$; the second case uses that $\langle g_j, x_i-x_j \rangle$ is zero when $j\geq i$ and nonpositive when $j<i$; and the third case uses that $g_\star=0$.

    Finally, suppose $i=\star$. Then,
    \begin{align*}
        j\in\{0,\dots,n-1\} \ &: \ f_\star = t \geq f_j + \langle g_j, y -x_j\rangle = f_j + \langle g_j, x_\star -x_j\rangle , \text{ and}\\
        j\in\{n,\dots,N\} \ &: \ f_\star = f_{n-1/2} - \Theta_n = f_j + \left\langle Me_j, -\frac{f_{n-1/2} - t}{M}e_j \right\rangle = f_j + \langle g_j, x_\star-x_j \rangle.
    \end{align*}    
    The first case uses the first constraint in~\eqref{eq:Kelley-solve} and the orthogonality of $g_j$ to $x_\star - y$.
    
    We deduce that $f_\cH(x_i) = f_i$ and $g_i \in \partial f_\cH(x_i)$ for all $i\in \cI_\star$.
    Since $g_\star=0\in\partial f_\cH(x_\star)$, $x_\star$ must be a minimizer of $f_\cH$. 
    It remains to verify the needed distance from initialization to a minimizer bound:
    $$\|x_0-x_\star\|^2 = \|y - x_0\|^2 + (N-n+1) \left(\frac{f_{n-1/2}-t}{M}\right)^2 \leq \|y - x_0\|^2 + (N-n+1)\zeta^2 \leq R^2, $$
    where the equality follows from orthogonality of each $e_i$ in our definition of $x_\star$ and the two inequalities use the constraints respectively $f_{n-1/2} - M\zeta \leq t$ and $\|y-x_0\|^2 + (N-n+1)\zeta^2 \leq R^2$.
\end{proof}

\begin{proposition}
\label{prop:zero_chain_Kelley}
Fix $j\in\{n,\dots,N-1\}$ and let $x\in x_0+\mathrm{span}\{g_0,\dots,g_{j-1}\}$. Then
\[
\{g_0, \dots, g_j\}\cap \partial f_{\cH} (x) \neq \emptyset.
\]
% In particular, an adversarial subgradient oracle can return either a previously seen subgradient or $g_j$ at round $j$.
\end{proposition}
\begin{proof}
Consider an $x\in x_0+\mathrm{span}\{g_0,\dots,g_{j-1}\}$.
Let $i\in\cI_\star$ denote an active component of $f_\cH(x)$ in the definition \eqref{eq:hard_func_Kelley}.
If $i<j$, then $g_i\in\partial f_\cH(x)$ and we are done.
Otherwise, suppose $i\geq j$. Note by orthogonality, that for every $i\geq j$, $f_i+\langle g_i,x-x_i\rangle
=f_i=f_{n-1/2}.$
% it must be the case then that every $i\geq j$ attains $f_\cH$.
As a result, $g_j\in\partial f_\cH(x)$.
\end{proof}

We are now ready to prove Theorem~\ref{thm:Kelley-perfection}.
\begin{proof}[Proof of Theorem~\ref{thm:Kelley-perfection}]
By Proposition~\ref{prop:feasibility-Kelley}, the instance $(f_\cH,x_0)\in\cF_{M,R}^\cH$.
Now consider an arbitrary deterministic subgradient-span method $A\in\cA^\cH$ responsible for producing iterates $x_n^A, x_{n+1}^A, \dots, x_N^A$.
By the subgradient-span condition, $x_n^A \in x_0 + \spann\{g_0,\dots,g_{n-1}\}$.
Inductively, suppose that for $n\leq j < N$, $x_j^A \in x_0 + \spann\{g_0,\dots, g_{j-1}\}$.
By Proposition~\ref{prop:zero_chain_Kelley}, an adversarial oracle could provide some subgradient in $\{g_0,\dots,g_{j}\}$, so that by the subgradient-span condition, $x_{j+1}^A\in x_0 + \spann\{g_0,\dots,g_j\}$.
We deduce that $x_N^A \in x_0 + \spann\{g_0,\dots,g_{N-1}\}$.
As a result,
    \begin{equation*}
        f_\cH(x_N^A) \geq f_N + \langle g_N, x_N^A - x_N\rangle = f_N = f_\star + (f_{n-1/2}-t) = f_\star + \Theta_n,
    \end{equation*}
    where the inequality restricts to the $N$th affine term defining $f_\cH$ and the first equality uses orthogonality of $g_N=Me_N$ to $\mathrm{span}\{g_0,\dots, g_{N-1}\}$ which contains $x_N^A - x_N$.
\end{proof}
\section{The Subgame Perfect Proximal Point Algorithm (SPPPA)} \label{sec:SPPPA}
We next design a subgame perfect algorithm in the setting of nonsmooth convex optimization with a proximal operator oracle. 
Let $L_0,L_1,\dots,L_N$ denote a fixed sequence of parameters. 
Our family of algorithms in this setting consists of any algorithm of the form:
Given an initialization $x_0$, iterate for $n = 1,\dots,N$: query the proximal operator oracle at $x_{n-1}$ to receive
\begin{equation}
\begin{cases}
y_{n-1} = \mathrm{prox}_{L_{n-1},f}(x_{n-1})\\
g_{n-1} = L_{n-1}(x_{n-1} - y_{n-1}) \in\partial f(y_{n-1})\\
f_{n-1} = f(y_{n-1})
\end{cases}\label{eq:general-form-prox-methods}
\end{equation}
and then pick the next query point
\begin{equation*}
    x_n \in x_0 + \spann\{g_0,\dots,g_{n-1}\}
\end{equation*}
as a function of the first-order history $\{(x_i, f_i, g_i)\}_{i=0}^{n-1}$.
The fact that $g_{n-1}\in\partial f(y_{n-1})$ is an immediate consequence of the optimality condition governing the proximal step's computation. 
We caution the reader that
in the proximal operator setting, we measure objective values at $y_i$ (in contrast to the subgradient oracle setting, where we evaluated objective values at $x_i$). 

Define $y_\star = x_\star\in \argmin_{x} f(x)$, $f_\star = f(y_\star) = \min_x f(x)$, and $g_\star = 0 \in \partial f(y_\star)$.
Observing that $g_i \in\partial f(y_i)$ for all $i\in\cI_\star$, in the following, we now use
\begin{equation}
    \label{eq:prox_nonnegative}
    Q_{i,j} \coloneqq f_i - f_j - \langle g_j, y_i -y_j\rangle \geq 0
\end{equation}
for all $i,j\in\cI_\star$. Note this definition has replaced $x_i$ and $x_j$ by $y_i$ and $y_j$ to account for the forward evaluation nature of proximal operators.

In our design of proximal point methods, we will not assume either Lipschitz continuity of $f$ or a bound on $\|x_0-y_\star\|$. Instead, the target class of problem instances is all $(f,x_0)$ with $f$ closed, convex, proper, and attaining a minimum at some $y_\star$.
In particular, the algorithm and analysis in this section readily apply to \emph{constrained} nonsmooth convex optimization.
In this setting, a natural goal is to seek to reduce the \emph{normalized} suboptimality
$$ \frac{f(y_N) - f(y_\star)}{\tfrac{1}{2}\|x_0-y_\star\|^2}.$$
Indeed, as $\|x_0 - y_\star\|$ is allowed to be arbitrarily large, it is impossible to provide any bound on the \emph{unnormalized} suboptimality for any proximal point method.

A series of classic works by~\cite{guler1991convergence,guler1992new,Monteiro2013hybrid,Barre2023principle} have developed minimax optimal methods for this task (although their exact optimality was only recently proven in~\cite{OPTIsta}).
Below, in Section~\ref{subsec:OPPA}, we discuss a minimax
optimal algorithm, called the Optimized Proximal Point Algorithm (OPPA), and present an inductive proof of its convergence guarantee.
In Section~\ref{subsec:SPPPA-design}, we will introduce the Subgame Perfect Proximal Point Algorithm (SPPPA), a modification of OPPA that dynamically optimizes the inductive statement within the proof of OPPA using observed first-order information.
The remainder of Section~\ref{sec:SPPPA} will prove dynamic upper bounds on the convergence rate of SPPPA. Dynamic lower bounds will be presented in Section~\ref{sec:SPPPA-proof}, thereby showing that SPPPA is subgame perfect.

Note that the setting of proximal operators has substantial similarities to the setting of smooth convex optimization. In particular, if one fixes all $L_n=L$ to be constant, then proximal operator steps are exactly gradient steps on the objective function's Moreau envelope. Correspondingly, our developed method and analysis here closely mirror those of the subgame perfect gradient method of~\cite{SPGM}, recovering the core operations of SPGM in the special case of a constant proximal parameter sequence.

\subsection{The Existing Optimized Proximal Point Algorithm (OPPA)} \label{subsec:OPPA}
We now describe a minimax optimal method, which we refer to as the ``Optimized Proximal Point Algorithm'' (OPPA). This method was first proposed in the case of constant proximal parameters $L_n=L$ by~\cite{guler1992new}. Generalizations allowing for any $L_n$ as well as inexactness in the evaluation of proximal operators were given by~\cite{Monteiro2013hybrid} and~\cite{Barre2023principle}, which reduce to OPPA when proximal operator evaluations are exact.

% The OPPA is described in 

% \begin{algorithm}
% \caption{OPPA}
% \label{alg:oppa}
% Given $x_0\in\R^d$ and proximal operator oracle with parameters $L_0,L_1,\dots$
% \begin{itemize}
%     \item Initialize $\tau_0 = 2/L_0$
%     \item Let $y_0 = \mathrm{prox}_{L_0,f}(x_0)$ and $f_0 = f(y_0)$. Define $g_0 = L_0(x_0 - y_0)$.
%     \item Set $z_1 = x_0 - \tau_0 g_0$
%     \item For $n = 1,\dots,N$
%     \begin{itemize}
%         \item Set
%         \begin{align*}
%     \tau_n & = \tau_{n-1} + \frac{1}{L_n}(1 + \sqrt{1+2L_n\tau_{n-1}})\\
%     x_{n} &= \frac{\tau_{n-1}}{\tau_n} y_{n-1} + \frac{\tau_{n}-\tau_{n-1}}{\tau_n} z_n
%         \end{align*}
%         \item Query $y_{n} = \mathrm{prox}_{L_{n},f}(x_{n})$ and $g_{n} = L_{n}(x_{n} - y_{n}) \in\partial f(y_{n})$
%         \item Set $z_{n+1} = z_n - (\tau_n - \tau_{n-1})g_n$
%     \end{itemize}
% \end{itemize}
% \end{algorithm}

OPPA maintains four sequences of iterates $x_n\in\R^d$, $y_n\in\R^d$, $z_n\in\R^d$ and $\tau_n\in\R$ via the following induction:
$x_0$ is given;
initialize $\tau_0 = 2/L_0$ 
% and take $x_0$ to be the initializer;
and  $z_1 = x_0 - \tau_0 g_0$; 
using the proximal operator oracle, inductively define for $n=1,2,\dots$
\begin{equation}\label{eq:OPPA}
\begin{cases}
    y_{n-1} &= \mathrm{prox}_{L_{n-1},f}(x_{n-1})\\
    g_{n-1} &= L_{n-1}(x_{n-1} - y_{n-1}) \in\partial f(y_{n-1}) \\
    \tau_n & = \tau_{n-1} + \frac{1}{L_n}(1 + \sqrt{1+2L_n\tau_{n-1}})\\
    x_{n} &= \frac{\tau_{n-1}}{\tau_n} y_{n-1} + \frac{\tau_{n}-\tau_{n-1}}{\tau_n} z_n\\
    z_{n+1} &= z_n - (\tau_{n}-\tau_{n-1}) g_n 
\end{cases}.
\end{equation}
It will be conceptually useful to view $z_{n+1}$ as being defined in the $n$th iteration, although it cannot explicitly be computed until the $(n+1)$th iteration, after $g_n$ is observed.
% Additionally, because $x_n$ can be recovered from $y_n$ and $g_n$, we may think of OPPA as only maintaining the sequences $y_n\in\R^d$, $z_n\in\R^d$ and $\tau_n\in\R$.
OPPA's design and proof strategy maintains a specific inductive hypothesis on its iterate sequences.
Our design of SPPPA and its proof will use these same ideas.

For $n\geq 0$, define the expression
\begin{equation}
    \label{eq:oppa_induction}
    H_n \coloneqq \tau_n(f_\star - f_n)  - \frac{1}{2}\|z_{n+1} - y_\star\|^2+ \frac{1}{2}\|x_0 - y_\star\|^2.
\end{equation}

% The first lemma below establishes the base case and the second lemma establishes the inductive statement.
% The first lemma below establishes that \eqref{eq:oppa_induction} holds for $(y_0, \tau_0, z_1)$. The second lemma below shows that if $(y,\tau,z)$ satisfies \eqref{eq:oppa_induction}, then OPPA's update rule \eqref{eq:OPPA} applied to this triple produces a new triple also satisfying \eqref{eq:oppa_induction}.

\begin{lemma}
\label{lem:OPPA-base-case}
For any closed convex proper $f$ with minimizer $x_\star=y_\star\in\R^d$ and any $x_0\in\R^d$, $H_0\geq 0$.
\end{lemma}
\begin{proof}
By definition,
$x_0 = y_0 + g_0 / L_0$, $z_1 = x_0 - 2g_0/L_0 = y_0 - g_0/L_0$, and $\tau_0 = \frac{2}{L_0}$.
Thus,
\begin{align*}
    H_0 &=\frac{2}{L_0}(f_\star - f_0)
    - \frac{1}{2}\left\|
    y_0 - x_\star -  g_0/L_0\right\|^2
    + 
    \frac{1}{2}\|y_0 - x_\star + g_0/L_0\|^2
    \\
    &= \frac{2}{L_0}(f_\star - f_0) + \frac{2}{L_0}\langle 
    g_0, y_0 - x_\star
    \rangle = \frac{2}{L_0} Q_{\star,0} \geq 0,
\end{align*}
where the last inequality follows from \eqref{eq:prox_nonnegative}.
\end{proof}

\begin{lemma}
\label{lem:OPPA_induction}
Suppose $f$ is closed convex proper with minimizer $x_\star = y_\star\in\R^d$, $m\in[0,n-1]$, $\tau'\in\R$, and $z'\in\R^d$ satisfy $H' \coloneqq \tau'(f_\star - f_m)  - \frac{1}{2}\|z' - y_\star\|^2+ \frac{1}{2}\|x_0 - y_\star\|^2\geq 0.$
Define
\begin{gather*}
    \tau_n  = \tau' + \frac{1}{L_n}(1 + \sqrt{1+2L_n\tau'}),\quad
    x_n = \frac{\tau'}{\tau_n} y_m + \frac{\tau_n -\tau'}{\tau_n} z',\quad\text{and}\quad
    z_{n+1} = z' - (\tau_n-\tau') g_n.
\end{gather*}
Then, $H_n\geq 0$.
\end{lemma}
\begin{proof}
It suffices to check the identity $H_n = H' + (\tau_n - \tau')Q_{\star,n} + \tau' Q_{m,n}$, which is nonnegative as it is the sum of three nonnegative terms. We verify this identity in two steps. First, consider the $f_i$ terms. Among these, the left and right-hand sides agree as
\begin{align*}
	\tau_n(f_\star - f_n) = \left[\tau'(f_\star-f_m)\right] + \left[(\tau_n-\tau')(f_\star - f_n)\right] + \left[\tau' (f_m-f_n)\right].
\end{align*}
Second, consider the remaining quadratic terms. Among these, the two sides agree as
\begin{align*}
	&- \frac{1}{2}\|z_{n+1} - y_\star\|^2+ \frac{1}{2}\|x_0 - y_\star\|^2\\
	&= \left[- \frac{1}{2}\|z' - y_\star\|^2 + \frac{1}{2}\|x_0 - y_\star\|^2\right] -\frac{1}{2}\|z_{n+1} - z'\|^2 - \langle z_{n+1} - z', z' - y_\star\rangle \\
	&= \left[- \frac{1}{2}\|z' - y_\star\|^2 + \frac{1}{2}\|x_0 - y_\star\|^2\right] - \frac{\tau_n}{L_n}\|g_n\|^2 + \langle (\tau_n-\tau')g_n, \frac{\tau_n}{\tau_n-\tau'} x_n - \frac{\tau'}{\tau_n-\tau'} y_m - y_\star\rangle \\ 
	&= \left[- \frac{1}{2}\|z' - y_\star\|^2 + \frac{1}{2}\|x_0 - y_\star\|^2\right] + (\tau_n-\tau')\left[-\langle g_n, y_\star - y_n\rangle\right] + \tau'\left[-\langle g_n, y_m - y_n\rangle\right]
\end{align*}
where the first line expands the square $\|z_{n+1} - y_\star\|^2$, the second plugs in the definitions of $z_{n+1}$ and $z'$ and notes $\tau_n/L_n=(\tau_n -\tau')^2/2$, and the third combines terms, noting $x_n-y_n=g_n/L_n$.
\end{proof}

Applying the above lemma with $m = n-1$, $\tau' = \tau_{n-1}$ and $z' = z_n$, we deduce that $H_n\geq 0$ for all $n\geq 0$. This leads us to OPPA's convergence guarantee:
\begin{equation}
    \label{eq:oppa_rate}
    \tau_N(f_N - f_\star) \leq \tau_N(f_N - f_\star) + \frac{1}{2}\|z_{N+1}-x_\star\|^2 \leq \frac{1}{2}\|x_0 - x_\star\|^2.
\end{equation}

To make this convergence guarantee concrete, if $L_n=L$ is constant, then $\tau_N=N^2/2L + o(N^2)$. Thus, the bound becomes $f(y_N) - f(x_\star) \leq L\|x_0-x_\star\|^2/N^2$ up to lower order terms. It has long been known that this $O(1/N^2)$ is the optimal order of convergence (see~\cite{nemirovski1991optimality,nemirovsky1992information}). Many works following~\cite{nesterov2003Introductory,drori2017LowerBoundOGM,carmon2020stationary1,drori2020complexity,carmon2021stationary2,dragomir2021optimal} have built lower bounding theory for first-order methods. Recently,~\cite{OPTIsta} extended the zero-chain framework of~\cite{drori2022LowerBoundITEM}, providing an exactly matching lower bound for OPPA, proving its optimality.

We emphasize OPPA only accesses $L_n$ at iteration $n$. Even so, OPPA achieves the optimal convergence guarantee among algorithms that receive the \emph{entire} sequence $L_0,L_1,\dots$ up front. Our algorithm SPPPA will share this same property.

\subsection{Design of SPPPA} \label{subsec:SPPPA-design}

SPPPA is presented in Algorithm~\ref{alg:SPPPA}.
Similar to OPPA, % We propose our subgame perfect proximal point algorithm as a dynamic method of the form~\eqref{eq:general-form-prox-methods} that utilizes the history of subgradients seen.
SPPPA will maintain iterate sequences $x_n,y_n,z_n$ (and a scalar sequence $\tau_n$) so that the expression $H_n$ (defined in \eqref{eq:oppa_induction}) is nonnegative for all $n\geq 0$.
We will adopt OPPA's initialization so that $H_0\geq 0$ (see Lemma~\ref{lem:OPPA-base-case}).
The key conceptual difference between OPPA and SPPPA is that SPPPA will, in each iteration $n$, 
apply Lemma~\ref{lem:OPPA_induction} to a choice of $m\in[0,n-1]$, $\tau'$, and $z'$ determined by the observed first-order responses. This contrasts with OPPA, where $m,\tau',z'$ are always set to $n-1$, $\tau_{n-1}$, and $z_{n}$ respectively.

% Our subgame perfect proximal point algorithm proceeds in a similar fashion to OPPA, except it seeks an optimized selection of $H'$ at each step rather than immediately setting $\tau'=\tau_{n-1}, z'=z_n, m=n-1$. The following subsections develop a tractable family of intermediate nonnegative hypotheses that could work as replacements, identify an objective to optimize over hypotheses, and finally formalize SPPPA as dynamically refining the hypothesis of OPPA as a function of the first-order method available at every iteration.

\subsubsection{The Planning Subproblem} \label{subsubsec:SPPPA-feasibility}

We now develop a tractable subproblem, the \emph{planning subproblem}, for optimizing $m,\tau',z'$.
Throughout this subsection, we will assume SPPPA is at iteration $n$ and has observed/constructed
$\{(y_i,f_i,g_i,z_{i+1},\tau_i)\}_{i=0}^{n-1}$ with $f_i=f(y_i)$, $g_i\in \partial f(y_i)$. Furthermore, we may assume by induction that $H_i\geq 0$ for all $0\leq i\leq n-1$.

Fix an arbitrary $m\in\argmin_{i\in[0,n-1]}f_i$ and define the following matrices and vectors
\begin{gather}
    Z=\bigl[z_{1}-x_0\ \ \cdots\ \ z_{n}-x_0\bigr]\in\R^{d\times n},\quad
    G=\bigl[g_0\ \ \cdots\ \ g_{n-1}\bigr]\in\R^{d\times n}, \label{eq:core-SPPPA-quantities}\\
    \tau=(\tau_0,\dots,\tau_{n-1})^\top,\quad
    f=(f_0,\dots,f_{n-1})^\top, \nonumber
\end{gather}
and as additional helpful quantities
\begin{equation}
    q=(q_0,\dots,q_{n-1})^\top,\quad a=(a_0,\dots,a_{n-1})^\top,\quad b=(b_0,\dots,b_{n-1})^\top \label{eq:helper-SPPPA-quantities}
\end{equation}
where $q_i:=f_i-\langle g_i,y_i-x_0\rangle$, $a_i:=\tfrac12\|z_{i+1}-x_0\|^2+\tau_i(f_i-f_m)$, and $b_i:=q_i-f_m$.
% Here we seek to construct nonnegative intermediate hypotheses of the form
% $$ H' = \tau'\big(f_\star-f_m\big) + \tfrac{1}{2}\|x_0-y_\star\|^2-\tfrac{1}{2}\|z' - y_\star\|^2 \geq 0 $$
% at a fixed iteration $1\leq n \leq N$ given a set of observed values $\{(y_i,f_i,g_i,z_{i+1},\tau_i)\}_{i=0}^{n-1}$ with $f_i=f(y_i)$, $g_i\in \partial f(y_i)$, and with all past inductive hypotheses nonnegative
% $$ H_i = \tau_i\big(f_\star-f_i\big) + \tfrac{1}{2}\|x_0-y_\star\|^2-\tfrac{1}{2}\|z_{i+1} - y_\star\|^2 $$
% for all $0\leq i\leq n-1$. 
Our goal is to find $\tau'$ and $z'$ so that $H'\geq 0$. We will attempt to certify the nonnegativity of $H'$ by writing it as
\begin{equation}\label{eq:SPPPA-decomposition}
    H' = \sum_{i=0}^{n-1}\mu_i H_i + \sum_{i=0}^{n-1}\lambda_{\star,i}Q_{\star,i} + \varepsilon
\end{equation}
for some nonnegative $\mu,\lambda_\star\in\R^n$ and $\varepsilon \in\R$.
% The following lemma characterizes when the above equality holds.
% provides an exact criteria for setting these values to generate a nonnegative hypothesis $H'$.
\begin{lemma}\label{lem:SPPPA-feasible-hypothesis}
    For any vectors $\mu,\lambda_\star \in\mathbb{R}^n$, the identity~\eqref{eq:SPPPA-decomposition} holds if
    \begin{gather*}
    \tau'=\langle\tau,\mu\rangle+\langle\mathbf 1,\lambda_\star\rangle,\qquad  
    z'=x_0+Z\mu-G\lambda_\star, \qquad 
    \varepsilon = \langle \mu,a\rangle+\langle \lambda_\star,b\rangle - \tfrac{1}{2}\|Z\mu - G\lambda_\star\|^2.
    \end{gather*}
    Hence, if $\mu,\lambda_\star \geq 0$ and $\varepsilon\geq 0$, then $H'\geq 0$.
\end{lemma}
\begin{proof}
    Verifying this identity simply corresponds to expanding the right-hand side and collecting like terms. For completeness, we present this verification below:

    Plugging in the definition $H_i=\tau_i(f_\star-f_i)+\tfrac{1}{2}\|x_0-y_\star\|^2-\tfrac{1}{2}\|z_{i+1}-y_\star\|^2$ and noting $\sum_i\mu_i(x_0-z_{i+1})=-Z\mu$, the first term in our claimed decomposition equals
    \begin{align*}
    \sum_i \mu_i H_i
    &=\langle\tau,\mu\rangle f_\star-\sum_i \mu_i\tau_i f_i
    -\tfrac{1}{2}\sum_i \mu_i\|z_{i+1}-x_0\|^2
    +\big\langle y_\star-x_0,\ Z\mu\big\rangle
    \\
    &= \langle\tau,\mu\rangle (f_\star-f_m) - \langle \mu, a\rangle
    +\big\langle y_\star-x_0,\ Z\mu\big\rangle.
    \end{align*}
    Using $q_i=f_i-\langle g_i,y_i-x_0\rangle$ and so $Q_{\star,i}=f_\star-f_i-\langle g_i,y_\star-y_i\rangle=f_\star-q_i-\langle g_i,y_\star-x_0\rangle,$ the second term in our claimed decomposition equals
    \begin{align*}
    \sum_i \lambda_{\star,i} Q_{\star,i}
    &=\langle\mathbf 1,\lambda_\star\rangle f_\star
    -\langle \lambda_\star,q\rangle
    -\big\langle y_\star-x_0,\ G\lambda_\star\big\rangle\\
    &=\langle\mathbf 1,\lambda_\star\rangle (f_\star-f_m)
    -\langle \lambda_\star,b\rangle
    -\big\langle y_\star-x_0,\ G\lambda_\star\big\rangle.
    \end{align*}
    Summing these two expressions with $\varepsilon = \langle \mu,a\rangle+\langle \lambda_\star,b\rangle-\tfrac{1}{2}\|Z\mu-G\lambda_\star\|^2$ gives
    \begin{align*}
        &\sum_{i=0}^{n-1}\mu_i H_i + \sum_{i=0}^{n-1}\lambda_{\star,i} Q_{\star,i} + \varepsilon\\
        &=(\langle\tau,\mu\rangle+\langle\mathbf 1,\lambda_\star\rangle)(f_\star - f_m) + \big\langle y_\star-x_0,\ Z\mu - G\lambda_\star\big\rangle -\tfrac{1}{2}\|Z\mu-G\lambda_\star\|^2\\
        &=(\langle\tau,\mu\rangle+\langle\mathbf 1,\lambda_\star\rangle)(f_\star - f_m) +\tfrac{1}{2}\|x_0-y_\star\|^2-\tfrac{1}{2}\|x_0 + Z\mu - G\lambda_\star -y_\star\|^2.
    \end{align*}
    Plugging in the chosen values of $\tau', z'$ from the lemma statement, this is exactly $H'$. The final conclusion that $H'$ must be nonnegative whenever $\mu,\lambda_\star \geq 0$ and $\varepsilon\geq 0$ follows from the fact that this decomposition shows $H'$ is then equal to a sum of nonnegative quantities.
\end{proof}
% We see that the set of $\mu,\lambda_\star$ that generate a valid nonnegative hypothesis $H'$ is a simple convex set, defined by nonnegativity constraints and one second-order cone representable constraint.

% The goal in the ``planning subproblem ''
The objective function in the planning subproblem is to maximize the value of $\tau'$. Hence, the planning subproblem can be written as
% Our goal is to maximize. Hence, our ``planning subproblem'' to optimize our intermediate hypothesis is to solve
\begin{align}
\label{eq:SPPPA-solve}
\tau' = \sup_{\mu,\lambda_\star\in\mathbb{R}^n_{\ge 0}}\left\{\langle \tau,\mu \rangle + \langle \mathbf{1},\lambda_\star\rangle:\ \tfrac{1}{2}\|Z\mu - G\lambda_\star\|^2\le \langle \mu,a\rangle + \langle \lambda_\star, b\rangle\right\}.
\end{align}
This is a simple convex optimization problem, independent of the ambient dimension $d$, optimizing a linear function over a feasible region given by nonnegativity and a single rotated second-order cone constraint.
Although the optimizer of \eqref{eq:SPPPA-solve} lacks a closed-form in general, linear optimization over a single convex quadratic constraint (and nonnegativity) is quite standard and can be done using industrial solvers.

Note that $\mu=(0,\dots,0,1), \lambda_\star=(0,\dots,0)$
is a feasible solution in \eqref{eq:SPPPA-solve}, as the constraint becomes $\frac{1}{2}\| z_n - x_0\|^2 \leq \frac{1}{2}\|z_n - x_0\|^2 + \tau_{n-1}(f_{n-1} - f_m)$,
which holds by the assumption that $m\in\argmin_{i\in[0,n-1]}f_i$.
We deduce that $\tau' \geq \tau_{n-1}$.

On the other hand, if \eqref{eq:SPPPA-solve} has unbounded optimal value, then we claim $y_m\in\argmin_x f(x)$. Indeed, for any feasible $\mu,\lambda_\star$ in \eqref{eq:SPPPA-solve}, it holds that $H'\geq 0$. Thus, by rearranging, we have that
\begin{equation*}
    f(y_m) - f_\star \leq \frac{1}{2(\langle \tau,\mu\rangle + \langle \mathbf{1}, \lambda_\star\rangle)}\| x_0 - x_\star\|^2.
\end{equation*}

\begin{remark}
    The SPPPA planning subproblem and the SPGM planning subproblem for $L$-smooth convex optimization are \emph{nearly identical}. In the case where $L_i$ is constant and equal to $L$, the SPPPA subproblem replaces the quantities $f(x_i)- \frac{1}{2L}\|\nabla f(x_i)\|^2$ and  $x_i - \frac{1}{L}\nabla f(x_i)$ in the SPGM subproblem with $f(y_i)$ and $y_i$ in the definitions of $a$ and $b$. To handle nonconstant $L_i$, the SPPPA planning subproblem introduces rescalings in the definitions of $\tau, \tau', \lambda_\star$, making them independent of the proximal parameter sequence.
\end{remark}

\subsubsection{Subgame Perfect Proximal Point Algorithm} \label{subsubsec:SPPPA-optimization}

We now formally state SPPPA:
\begin{algorithm}[H]
    \caption{The Subgame Perfect Proximal Point Algorithm (SPPPA)}
    \label{alg:SPPPA}
    Given closed convex proper function $f$, initial iterate $x_0$, proximal parameter sequence $L_0,L_1,\dots$
    \begin{itemize}
        \item Define $\tau_0 = \frac{2}{L_0}$ and $z_1 = x_0 - \tau_0 g_0$
        \item For $n = 1,2,\dots$.
        \begin{itemize}
            \item Query $y_{n-1}=\mathrm{prox}_{L_{n-1},f}(x_{n-1})$ and $f_{n-1} = f(y_{n-1})$. Set $g_{n-1} = L_{n-1}(x_{n-1} - y_{n-1})$.
            \item Let $m \in \argmin_{i\in\{0,\dots,n-1\}}\left\{f(y_i)\right\}$. If \eqref{eq:SPPPA-solve} is unbounded, terminate and output $y_m$. Else, let $(\mu,\lambda_\star)$ be an optimal solution to \eqref{eq:SPPPA-solve} and set $\tau' = \langle\tau,\mu\rangle+\langle \mathbf{1},\lambda_{\star}\rangle$ and $z' = x_0 + Z\mu - G\lambda_\star$.
            \item Define $\begin{cases}\tau_n = \tau' + \frac{1}{L_n}(1 + \sqrt{1+2L_n\tau'})\\
            x_{n} = \frac{\tau'}{\tau_n} y_{m} + \frac{\tau_{n}-\tau'}{\tau_n} z'\\
            z_{n+1} = z' - (\tau_{n}-\tau') g_n.\end{cases}$
        \end{itemize}
    \end{itemize}
    \end{algorithm}
By 
Lemmas~\ref{lem:OPPA-base-case},~\ref{lem:OPPA_induction} and~\ref{lem:SPPPA-feasible-hypothesis}, this method maintains $H_n\geq 0$ at every iteration and hence guarantees
$$ \frac{f(y_n)-f(y_\star)}{\tfrac{1}{2}\|x_0-x_\star\|^2} \leq \frac{1}{\tau_n}.$$

In order to argue that SPPPA is subgame perfect, we will additionally need dynamic guarantees on $f(y_N) - f(y_\star)$ that can be made at earlier iterations $0\leq n\leq N$.
We will state these guarantees in terms of a doubly-indexed expression $\tau_{i,j}$.
First, define $\tau_{0,0},\dots,\tau_{0,N}$ to be the sequence generated by the OPPA recurrence: $\tau_{0,0} = 2/L_0$ and for all $i\in[1,N]$, $\tau_{0,i} = \tau_{0,i-1} + \frac{1}{L_{i}}(1 + \sqrt{1+2L_{i}\tau_{0,i-1}}).$
Next, for every fixed $n\in[1,N]$, let $\tilde\tau$ denote the optimal value of \eqref{eq:SPPPA-solve} in iteration $n$, and define
\begin{equation}\label{eq:tau-recurrence}
    \begin{cases}
    \tau_{n,n-1}=\tilde\tau\\
    \tau_{n,i} = \tau_{n,i-1} + \frac{1}{L_{i}}(1 + \sqrt{1+2L_{i}\tau_{n,i-1}}).
\end{cases}
\end{equation}
Note that the $\tau_n$ sequence maintained in SPPPA coincides with the sequence $\tau_{n,n}$.

\begin{theorem}\label{thm:SPPPA-rate}
    For any closed convex proper $f$ and $x_0\in\mathbb{R}^d$ and fixed sequence of proximal parameters $L_0,\dots,L_{N}$, SPPPA guarantees
    $$\frac{f(y_N)-f_\star}{\tfrac{1}{2}\|x_0-x_\star\|^2} \leq \Psi_N \leq \dots \leq \Psi_0,$$
    where $\Psi_n = 1/\tau_{n,N}$ is the guarantee of SPPPA based on the first-order responses seen up to iteration $n$ and $\Psi_0$ is the minimax optimal guarantee ensured by OPPA.
\end{theorem}
\begin{proof}
Recall that at iteration $n$, the optimal value of $\tau'$ in \eqref{eq:SPPPA-solve} is at least $\tau_{n-1}$. This is equivalent to saying that $\tau_{n,n-1} \geq \tau_{n-1,n-1}$ for all $n=1,\dots,N$.
Now, observe that for any $L>0$, the expression
    $\tau + \frac{1}{L}(1+\sqrt{1+2L\tau})$
is an increasing function of $\tau$. We deduce that for any $n\in[1,N]$, 
\begin{equation*}
    \tau_{n,n} = \tau_{n,n-1} + \frac{1}{L_n}(1 + \sqrt{1+2L_n\tau_{n,n-1}})
    \geq \tau_{n-1,n-1} + \frac{1}{L_n}(1 + \sqrt{1+2L_n\tau_{n-1,n-1}})
    = \tau_{n-1,n}.
\end{equation*}
We can chain this argument to get:
\begin{equation*}
    \tau_{n,n} \geq \tau_{n-1,n-1} + \frac{1}{L_n}(1 + \sqrt{1+2L_n\tau_{n-1,n-1}})
    \geq \tau_{n-2,n-1} + \frac{1}{L_n}(1 + \sqrt{1+2L_n\tau_{n-2,n-1}})
    = \tau_{n-2,n}.
\end{equation*}
Repeating this argument shows that
% \begin{align*}
    $\tau_{0,N}\leq \tau_{1,N}\leq \dots\leq \tau_{N,N}$,
% \end{align*}
thereby completing the argument.
\end{proof}

The SPPPA algorithm can be viewed as a bundle method, maintaining a main sequence of iterates $y_m$ used to compute each candidate step $y_n$. At each iteration, if $f(y_n)\geq f(y_m)$, the main sequence remains unchanged (a ``null'' step) but the method's memory bundle is improved from subgradient information at $y_n$. If $f(y_n)< f(y_m)$, the main sequence updates to $y_m \leftarrow y_n$ (a ``serious'' step). Our analysis does not require distinction between these two types of steps, so we define SPPPA by a single loop. However, for comparison with existing bundle methods, it may be useful to view it as a series of null steps growing a memory followed by a serious step making objective progress.

In the next section, we will show that SPPPA is subgame perfect. Section~\ref{sec:pragmatics} then notes several practical improvements to this algorithm that follow naturally: extending this method to allow adaptive selections of $L_n$, inexact proximal operators, and usage of a limited memory.

\section{Subgame Perfection of SPPPA} \label{sec:SPPPA-proof}
This section constructs, for any given history of first-order information revealed before the $n$th iteration, a worst–case convex function for which no method of the form~\eqref{eq:general-form-prox-methods} can outperform the SPPPA guarantee. Our argument leverages the zero-chain construction ideas of~\cite{drori2022LowerBoundITEM} that were extended to proximal settings by~\cite{OPTIsta}. Our process for dynamically constructing lower bounds parallels that of the smooth convex setting in~\cite{SPGM}.

\subsection{Dynamic Construction of Candidate Hard Problem Instance}
\label{subsec:spppa_lowerbound}

Fix an $n\in[1,N]$ in this section. 
Let
$\cH=\{(x_i,f_i,g_i,y_i,\tau_i,z_{i+1})\}_{i=0}^{n-1}$
be the iterates and responses produced/observed by SPPPA with proximal parameters $(L_i)_{i\ge 0}$, where $y_i=\mathrm{prox}_{L_i,f}(x_i)$ and $g_i=L_i(x_i-y_i)\in\partial f(y_i)$, and $f_i=f(y_i)$. Note that by construction, for all $i,j\in\{0,\dots,n-1\}$,
\[
Q_{i,j}=f(y_i)-f(y_j)-\langle g_j,y_i-y_j\rangle\ge 0,
\]
and
\[
H_i=\tau_i(f(y_\star)-f(y_i))+\tfrac12\|x_0-y_\star\|^2-\tfrac12\|z_{i+1}-y_\star\|^2\ge 0.
\]

We will take our hard function $f$ to be a max-of-affine function
\begin{equation}
    \label{eq:hard_func_SPPPA}
    f_\cH(x) = \max_{i\in\cI_\star}\{f_i + \langle g_i, x - y_i\rangle\},
\end{equation}
where the tuples $(f_i, g_i, y_i)$ come from the given history of first-order responses for $i = 0,\dots,n-1$, and need to be constructed for $i=n,\dots,N,\star$.

\subsubsection{Dual of the Planning Subproblem}
As the guarantee of SPPPA depends on the optimal value of the planning subproblem, it is natural to use a dual optimal solution to construct lower bounds.
We continue using the notation defined in~\eqref{eq:core-SPPPA-quantities} and~\eqref{eq:helper-SPPPA-quantities}.
Let $(\mu,\lambda_\star)$ denote the maximizers of~\eqref{eq:SPPPA-solve} computed by SPPPA in its $n$th iteration, generating values $\tau'=\langle\tau,\mu\rangle+\langle \mathbf{1},\lambda_\star\rangle$ and $z'=x_0+Z\mu-G\lambda_\star$. We derive the dual to \eqref{eq:SPPPA-solve} below.

\begin{lemma}
\label{lem:dual_prox_full}
The dual of \eqref{eq:SPPPA-solve} is
\begin{equation}
\label{eq:phi_prox_dual_full}
\inf_{\xi>0,\ w\in\R^d}\left\{
\frac{1}{2\xi} \|w\|^2\ :\
\tau+\xi a - Z^\top w \le 0,\ \ 
\mathbf{1}+\xi b + G^\top w \le 0
\right\}.
\end{equation}
If the supremum in \eqref{eq:SPPPA-solve} is finite, it is attained and strong duality holds. For any primal/dual optimizers $(\mu,\lambda_\star)$ and $(\xi,w)$, $w=\xi(z'-x_0).$
% for some scalar $f_\star\in\R$.
\end{lemma}
\begin{proof}
Consider the Lagrangian with multiplier $\xi\ge 0$:
\[
\cL(\mu,\lambda_\star;\xi)
=\langle\tau,\mu\rangle+\langle \mathbf{1},\lambda_\star\rangle
+\xi\bigl(\langle\mu,a\rangle+\langle\lambda_\star,b\rangle
-\tfrac12\|Z\mu-G\lambda_\star\|^2\bigr),
\quad \mu,\lambda_\star\ge 0.
\]
Using the Fenchel identity
\(
-\tfrac{\xi}{2}\|t\|^2
=\inf_{w\in\R^d}\bigl\{\tfrac{1}{2\xi}\|w\|^2-\langle w,t\rangle\bigr\}
\)
(valid for $\xi\ge 0$), we get
\[
\cL(\mu,\lambda_\star;\xi)
=\inf_{w\in\R^d}
\Bigl\{
\tfrac{1}{2\xi}\|w\|^2
+\langle \tau+\xi a-Z^\top w,\mu\rangle
+\langle \mathbf{1}+\xi b+G^\top w,\lambda_\star\rangle
\Bigr\}.
\]
Maximizing over $\mu,\lambda_\star\ge 0$ yields finiteness if and only if $\tau+\xi a-Z^\top w\le 0$ and $\mathbf{1}+\xi b+G^\top w\le 0$
in which case the supremum over $(\mu,\lambda_\star)$ equals $0$. Then algebraic simplification gives the claimed dual program.
If the supremum in \eqref{eq:SPPPA-solve} is finite, the recession directions that would make it $+\infty$ are excluded, which forces $\xi>0$; standard conic/Fenchel duality (and closedness of the feasible set) then gives strong duality and attainment on both sides.
By the Fenchel step, at any primal/dual optimizers \((\mu,\lambda_\star)\) and \((\xi,w)\) we have $w  =  \xi (Z\mu - G\lambda_\star) = \xi (z'-x_0).$
% Complementary slackness for the quadratic constraint (and \(\xi>0\)) gives
% \[
% \langle \mu,a\rangle+\langle \lambda_\star,b\rangle
% = \tfrac12\|Z\mu - G\lambda_\star\|^2.\qedhere
% \]
% and for the nonnegativity constraints (componentwise, for all \(i\))
% \[
% \big(\tau+\xi a - Z^\top w\big)_i \mu_i=0,
% \qquad
% \big(\mathbf{1}+\xi b + G^\top w\big)_i \lambda_{\star,i}=0.\qedhere
% \]
% It is convenient to parameterize \(\xi\) via a scalar \(f_\star\in\R\) by
% \[
% \frac{1}{\xi} = f_m - f_\star \quad(\text{equivalently } f_\star=f_m-\tfrac{1}{\xi}). \qedhere
% \]
\end{proof}

It will be instructive to view the dual variable 
$w=\xi(z'-x_0)$ as parameterized by $z'$ and $\xi = \frac{1}{f_m - f_\star}$ as parameterized by the variable $f_\star$. This quantity will coincide with the optimal value of the hard function that we will construct in this next section, justifying the notation $f_\star$.
The constraints in the dual problem can be rewritten as constraints on $f_\star$ and $z'$:
for every $i\in\{0,\dots,n-1\}$,
\begin{gather}
f_\star \ \ge\ f_i + \langle g_i,\ z' - y_i\rangle
% \quad\text{(subgradient support at \(y_i\))}
,
\label{eq:first-dual-constraint-SPPPA}\\
\tau_i (f_\star - f_i) + \tfrac12\|x_0 - z'\|^2 - \tfrac12\|z_{i+1}-z'\|^2 \ \ge\ 0
% \quad\text{(potential nonnegativity with \(y_\star=z'\))}
. \label{eq:second-dual-constraint-SPPPA}
\end{gather}
To see this, substituting $w=\xi(z'-x_0)$ and $1/\xi=f_m-f_\star$ into the dual constraints gives
\begin{align*}
	0 
	&\ge 1+\xi b_i+\langle g_i,w\rangle
	=\xi\left[
	f_i-f_\star+\langle g_i,z'-y_i\rangle
	\right],\\
	0 
	&\ge \tau_i+\xi a_i-\langle z_{i+1}-x_0,w\rangle=\xi\left[
	\tau_i(f_i-f_\star)+\tfrac12\|z_{i+1}-x_0\|^2
	-\langle z_{i+1}-x_0,z'-x_0\rangle
	\right].
\end{align*}
From these, simple rearrangements give the claimed formulas.

\subsubsection{Construction of Our Dynamic Hard Instance}
% Given the observed history $\cH = \{(x_i,y_i,g_i,\tau_i,z_{i+1})\}_{i=0}^{n-1}$ of SPPPA prior to iteration $n$, we now dynamically construct a problem instance capable of providing a matching lower bound.
It suffices to consider the case where \eqref{eq:SPPPA-solve} is bounded as otherwise SPPPA outputs an exact minimizer of $f$.
Let $(\mu,\lambda_\star)$ denote optimizers and let
$(w,\xi)$ be optimizers of the primal and dual problems.
By Lemma~\ref{lem:dual_prox_full}, $w=\xi(z'-x_0)$. Let $e_n,\dots,e_N$ be unit vectors orthogonal to $\mathrm{span}\{g_0,\dots,g_{n-1}\}$ and mutually orthogonal. It is possible to pick these vectors under the assumption $d\geq N+1$.

Then we define {\it future} iterates and first-order observations for $i\in\{n,\dots, N\}$ as follows (first at $i=n$, then inductively for $i>n$)
\begin{align*}
    \tau_n &= \tau' + \frac{1}{L_n}\left(1+\sqrt{1+2L_n\tau'}\right) & \tau_i &= \tau_{i-1} + \frac{1}{L_i}\left(1+\sqrt{1+2L_i\tau_{i-1}}\right)\\
    x_n &= \frac{\tau'}{\tau_n}y_{m}+\frac{\tau_{n}-\tau'}{\tau_n}z' & x_i &= \frac{\tau_{i-1}}{\tau_i}y_{i-1}+\frac{\tau_{i}-\tau_{i-1}}{\tau_i}z_i\\
    g_n &= \sqrt{\frac{f_m-f_\star}{\tau_n - \tau'}} e_n & g_i &= \sqrt{\frac{f_{i-1}-f_\star}{\tau_i - \tau_{i-1}}}e_i\\
    f_n &= f_{m} - \frac{1}{L_n}\|g_n\|^2 & f_i &= f_{i-1} - \frac{1}{L_i}\|g_i\|^2\\
    y_n &= x_n - \frac{1}{L_n}g_n & y_i &= x_i - \frac{1}{L_i}g_i\\
    z_{n+1} &= z'- (\tau_n - \tau') g_n & z_{i+1} &= z_{i}-(\tau_{i} - \tau_{i-1}) g_{i}.
\end{align*}
For $i=\star$, set
$f_\star = f_m - \frac{1}{\xi}$, 
$g_\star=0$ and $y_\star=z_{N+1}$. 

Note that with these definitions, $\tau_{n,i} = \tau_i$ for all $i=n,\dots,N$ so that $\Psi_n = \frac{1}{\tau_{n,N}} = \frac{1}{\tau_N}$.

% Together these definitions provide us with values $\cH = \{(x_i,y_i,f_i,g_i,z_{i+1})\}_{i\in\cI}$ from which we define 
% \begin{equation}

% f_{\cH}(y)=\max_{i\in\cI}\{f_i+\langle g_i,y-y_i\rangle\}.
% \end{equation}

\subsection{Properties of the Candidate Hard Instance}
\label{subsec:properties-SPPPA-lowerbound}
Since the chosen directions $g_j$ are each mutually orthogonal and orthogonal to any directions seen in the history $\mathcal{H}$, we can conclude several additional orthogonality relations.
\begin{lemma}\label{lem:orthogonality}
	For each $j\in\{n,\dots,N\}$, the subgradient $g_j$ is orthogonal to each $g_i$ with $i\neq j$, each $y_i-x_0$ with $i<j$, each $x_i-x_0$ with $i\leq j$, each $z_i-x_0$ with $n+1\leq i\leq j$, and $z'-x_0$.
\end{lemma}
\begin{proof}
	The orthogonality of $g_j$ and $g_i$ with $j\geq n$ and $i\neq j$ follows from the defining orthogonality of the vectors $e_n,\dots,e_N$. To see the remaining orthogonality results, notice that $g_j$ is orthogonal to everything in the $\mathrm{span}\{g_0,\dots,g_{j-1}\}$. Since the considered problem data $\mathcal{H}$ is derived from a span-method, each $y_i-x_0$ and $x_i-x_0$ with $i<n$ lies in this span. The constructed future iterates each have $y_i - x_0 \in \mathrm{span}\{g_0,\dots,g_{j-1}\}$ and $x_i-x_0,z_i-x_0\in \mathrm{span}\{g_0,\dots,g_{i-1}\}$. Further, $z'-x_0\in \mathrm{span}\{g_0,\dots,g_{n-1}\}$. From these spans, the remaining orthogonalities all follow.
\end{proof}
The following three lemmas record further useful algebraic properties of our construction. Proofs of each of these results are given in the appendix for completeness.
\begin{lemma}\label{lem:inner_product_formulas}
	For $i\in\{n,\dots, N\}$, the following hold for each index $j<i$
	\begin{equation}\label{eq:innerprod}
		\langle g_j, y_i - y_j\rangle
		=\begin{cases} \langle g_j, z' - y_j\rangle + \frac{\tau'}{\tau_i}\,\langle g_j, y_m - z'\rangle & \text{if } j<n\\
			 -\frac{\tau_i - \tau_j}{\tau_i}\,\bigl(f_j - f_\star\bigr) &\text{if } j\geq n.
			\end{cases}
	\end{equation}
\end{lemma}

\begin{lemma}\label{lem:SPPPA-monotone}
	For $i\in\{n,\dots, N\}$, $\tau_i\bigl(f_i - f_\star\bigr)$ is nondecreasing and $\tau_n\bigl(f_n - f_\star\bigr) \ge \tau'(f_m - f_\star)$.
\end{lemma}

\begin{lemma}
\label{lem:tight_ineqs_SPPPA}
It holds that $Q_{m,n}=Q_{n,n+1} = \dots = Q_{N-1,N} = 0$.
For $i\in\{n,\dots, N\}$,
\[
Q_{\star,i}=0,\qquad H_i=\tau_i(f_\star -f_i)+\tfrac12\|x_0-y_\star\|^2-\tfrac12\|z_{i+1}-y_\star\|^2=0.
\]
% Additionally, $Q_{m,n} = 0$ and for all $i\in\{n+1,\dots,N\}$, $Q_{i-1,i} = 0$.
\end{lemma}

The following pair of propositions establish the key properties enabling $f_\cH$ to serve as a dynamic hard lower bounding instance for proving subgame perfection. First we show that indeed $f_\cH$ interpolates the past observed data and the proposed future values in our definition. Second we show that this construction has a zero-chain property, preventing any proximal point-type method~\eqref{eq:general-form-prox-methods} from discovering more than one new direction $e_i$ per iteration.

\begin{proposition}
\label{prop:interpolating_prox_full}
The function $f_{\cH}$ is proper, closed, and convex, satisfying for every $i\in\cI_\star$, $f_{\cH}(y_i)=f_i$ and $g_i\in\partial f_{\cH}(y_i)$.
\end{proposition}
\begin{proof}
Note that $f_\cH$ must be closed, convex, and proper since it is defined as a finite maximum of affine functions. As occurred in our previous lower bound construction for KLM in Proposition~\ref{prop:feasibility-Kelley}, it suffices to verify nonnegativity of each $Q_{i,j}$ to verify $f_{\cH}(y_i)=f_i$ and $g_i\in\partial f_{\cH}(y_i)$. Again, we do this case-wise:

First consider $i\in\{0,\dots,n-1\}$. If $j\in\{0,\dots,n-1\}$ as well, then $Q_{i,j}\geq 0$ follows since the past observations were generated from some convex function.  If $j\in\{n,\dots,N\}$, then
\begin{align}
    Q_{i,j} &= f_i - \left(f_j + \langle g_j, y_i-y_j\rangle\right)\nonumber \\
    &=f_i - \left(f_\star - \langle g_j, y_\star-y_j\rangle + \langle g_j, y_i-y_j\rangle\right)\nonumber\\
    &=f_i - f_\star + \langle g_j, y_\star - y_i\rangle\nonumber\\
    &=f_i - f_\star - (\tau_{j}-\tau_{j-1})\|g_j\|^2\nonumber\\
    &=f_i - f_{j-1} \geq 0 \label{eq:first-main-case-SPPPA}
\end{align}
where the second equality uses that $Q_{\star,j}=0$ by Lemma~\ref{lem:tight_ineqs_SPPPA} and the fourth notes by Lemma~\ref{lem:orthogonality} that $g_j$ is orthogonal to $y_i-z'$, so $\langle g_j, y_\star - y_i\rangle = \langle g_j, y_\star - z'\rangle = - (\tau_{j}-\tau_{j-1})\|g_j\|^2$.
If $j=\star$, then
$$ Q_{i,\star} = f_i - \left(f_\star + \langle g_\star, y_i - y_\star\rangle\right) = f_i - f_\star = f_i - f_m + \frac{1}{\xi} \geq 0.$$

Next consider $i\in\{n,\dots, N\}$. If $j\in\{0,\dots,n-1\}$, then
\begin{align*}
    Q_{i,j}
    &= f_i - f_j - \langle g_j, y_i - y_j\rangle\\
    &= f_i - f_j - \Bigl(\langle g_j, z' - y_j\rangle + \frac{\tau'}{\tau_i}\langle g_j, y_m - z'\rangle\Bigr)\\
    &= f_i - \left(1-\frac{\tau'}{\tau_i}\right)\left(f_j + \langle g_j, z' - y_j\rangle\right)
    - \frac{\tau'}{\tau_i}\left(f_j + \langle g_j, y_m - y_j\rangle\right)\\
    &\geq f_i - \left(1-\frac{\tau'}{\tau_i}\right)f_\star - \frac{\tau'}{\tau_i} f_m\\
    &= \frac{\tau_i(f_i-f_\star) - \tau'(f_m-f_\star)}{\tau_i} \geq 0,
\end{align*}
where the second equality uses Lemma~\ref{lem:inner_product_formulas}, the first inequality uses~\eqref{eq:first-dual-constraint-SPPPA} and $Q_{m,j}\geq 0$ and the second inequality uses Lemma~\ref{lem:SPPPA-monotone}.
If $j\in\{n,\dots, i-1\}$, then
\begin{equation*}
		Q_{i,j} = f_i - f_j - \langle g_j, y_i - y_j\rangle
		= f_i - f_j + \frac{\tau_i - \tau_j}{\tau_i}(f_j - f_\star)
		= \frac{1}{\tau_i}\left(\tau_i(f_i - f_\star) - \tau_j(f_j - f_\star)\right) \geq 0,
\end{equation*}
where the second equality uses Lemma~\ref{lem:inner_product_formulas} and the final inequality follows from Lemma~\ref{lem:SPPPA-monotone}.
If $j\in\{i+1,\dots, N\}$, then this follows by the same reasoning presented in~\eqref{eq:first-main-case-SPPPA}.
If $j=\star$, then
$$ Q_{i,\star} = f_i - \left(f_\star + \langle g_\star, y_i - y_\star\rangle\right) = f_i - f_\star \geq 0.$$

Finally, consider $i=\star$. Then for $j\in\{0,\dots,n-1\}$, having $Q_{\star,j}\geq 0$ is precisely guaranteed by the first dual constraint~\eqref{eq:first-dual-constraint-SPPPA}. For $j\in\{n,\dots,N\}$, $Q_{\star,j}=0$ is guaranteed by Lemma~\ref{lem:tight_ineqs_SPPPA}.
\end{proof}

\begin{proposition}
\label{prop:zero_chain_prox_full}
Suppose $0\leq n\leq N$ and the history $\cH = \{(x_i,f_i,g_i, y_i,\tau_i,z_{i+1})\}_{i=0}^{n-1}$ is given. Let $f_\cH$ denote the function constructed in \eqref{eq:hard_func_SPPPA} and Section~\ref{subsec:spppa_lowerbound}.
Let $j\in\{n,\dots,N-1\}$. If 
% At any iteration $n$ and given history $\cH$, the function $f_{\cH}$ in \eqref{eq:hard_func_SPPPA} has, for any $j\in \{n,\dots,N-1\}$,
$x\in x_0+\mathrm{span}\{g_0,\dots,g_{j-1}\},$
then $\mathrm{prox}_{L_{j},f_{\cH}}(x) \in x_0+\mathrm{span}\{g_0,\dots,g_{j}\}.$
\end{proposition}
\begin{proof}
Fix $j\in \{n,\dots,N-1\}$ and define $\tilde f(x) = \max_{i\leq j}\{f_i + \langle g_i, x - y_i\rangle\} \leq f_\cH(x).$
Further fix an $x\in x_0 + \mathrm{span}\{g_0,\dots,g_{j-1}\}$.
Let $y = \mathrm{prox}_{L_{j},\tilde f}(x)$.
We know that $g = L_j(x-y)\in\partial \tilde f(y)$ is a convex combination of $\{g_0,\dots,g_j\}$.
Then, as $x\in x_0 + \mathrm{span}\{g_0,\dots,g_{j-1}\}$, we deduce that $y = x - g/L_j \in x_0 + \mathrm{span}\{g_0,\dots,g_{j}\}$.
% which necessarily lies in $x_0+\mathrm{span}\{g_0,\dots,g_{j}\}$ since $x_0+\mathrm{span}\{g_0,\dots,g_{j}\}$.
% In particular, $g = L_j(x-y)\in\partial \tilde f(y)$ must be a convex combination of $\{g_0,\dots, g_{j}\}$ since all subgradients of $\tilde f$ are.
Recall that we defined $g_j$ to be orthogonal to $g_i$ for all $i< j$. Thus, it follows that $\langle g_j, g\rangle\leq \|g_j\|^2$.
% attained with equality when $g=g_j$.
Hence for any $i>j$, we observe that
\begin{align*}
    \tilde f(y) & \geq f_j + \langle g_j, y-y_j\rangle\\ 
    & = f_j + \left\langle g_j, \left(x - \frac{1}{L_j}g\right) - \left(x_{j} - \frac{1}{L_j}g_j\right)\right\rangle\\
    & = f_j + \frac{1}{L_j}\left\langle g_j, g_j - g\right\rangle\\
    & \geq f_j\\
    & \geq f_i + \langle g_i, y_j - y_i\rangle\\
    & = f_i + \langle g_i, y - y_i\rangle,
\end{align*}
where the first line lower bounds $\tilde f$ by its final linear term, the second and third apply definitions and simplify, the fourth uses $\langle g_j, g \rangle \leq \|g_j\|^2$, the fifth uses $Q_{j,i}\geq 0$ and the last uses Lemma~\ref{lem:orthogonality}'s orthogonality of $g_i$ to $y_j-y$. Similarly, considering $i=\star$, we have that $\tilde f(y) \geq f_j \geq f_\star + \langle g_\star, y-y_\star\rangle$. We deduce that $f_{\cH}(y) = \tilde f(y)$.
Hence, for all $y'\in\mathbb{R}^d$, we have that
\begin{align*}
    f_\cH(y) + \frac{L_j}{2}\|y-x\|^2 = \tilde f(y) + \frac{L_j}{2}\|y-x\|^2 \leq \tilde f(y') + \frac{L_j}{2}\|y'-x\|^2 \leq f_\cH(y') + \frac{L_j}{2}\|y'-x\|^2,
\end{align*}
where the first inequality uses the definition of $y$ and the second inequality uses the fact that $\tilde f\leq f_\cH$ pointwise. We deduce that $\mathrm{prox}_{L_{j},f_{\cH}}(x) = y \in x_0+\mathrm{span}\{g_0,\dots,g_{j}\}$.
\end{proof}

\subsection{Proof of Subgame Perfection}
\label{subsec:SPPPA-perfection}

\begin{theorem}
\label{thm:SPPPA-perfection}
Assume $d\geq N+1$. Suppose $0\leq n \leq N$ and 
the history $\mathcal{H}$ is given.
Let $f_{\cH}$ denote the function constructed in~\eqref{eq:hard_func_SPPPA} and Section~\ref{subsec:spppa_lowerbound}. Then any method $A$ of the form~\eqref{eq:general-form-prox-methods} generating $x_0,\dots, x_{n-1}$ when applied to $f_{\cH}$ has terminal iterate $y_N^A $ satisfying
\[
\frac{f_{\cH}(y_N^A)-f_\star}{\tfrac{1}{2}\|x_0-y_\star\|^2} \ge \Psi_n.
\]
Consequently, combined with Theorem~\ref{thm:SPPPA-rate}, SPPPA is subgame perfect.
\end{theorem}
\begin{proof}
Let $x_j^A, y_j^A, g_j^A$ denote the sequence of iterates and subgradients generated by some method of the form~\eqref{eq:general-form-prox-methods} consistent with the given history $\cH$ prior to iteration $n$. Note, for $j\geq n$, that $x_j,y_j,g_j$ may be distinct from $x_j^A, y_j^A, g_j^A$.
By Proposition~\ref{prop:zero_chain_prox_full}, if $x_j^A\in x_0+\mathrm{span}\{g_0,\dots,g_{j-1}\}$ then 
$y_j^A=\mathrm{prox}_{L_j,f_{\cH}}(x_j^A)\in x_0+\mathrm{span}\{g_0,\dots,g_j\}$
and hence $g_j^A \in \mathrm{span}\{g_0,\dots,g_j\}$.
Since any method of the form~\eqref{eq:general-form-prox-methods} sets $x_{j+1}^A\in x_0+\mathrm{span}\{g_0^A,\dots,g_j^A\}$, induction over $j=n,\dots,N$ yields
\[
y_N^A\in x_0+\mathrm{span}\{g_0^A,\dots,g_N^A\}\subseteq x_0+\mathrm{span}\{g_0,\dots,g_N\}.
\]
Further, $g_N^A = L_N(x_N^A - y_N^A)\in\partial f_\cH(y_N^A)$ must lie in the convex hull of $\{g_0,\dots, g_N, g_\star\}$. Noting $g_N$ is orthogonal to all other $g_i$, we have $\langle g_N, g_N^A\rangle \leq \|g_N\|^2$.
Consequently,
\begin{align*}
    f_{\cH}(y_N^A)&\geq f_N+\langle g_N,y_N^A-y_N\rangle\\
    &= f_N+\left\langle g_N,\left(x_N^A - \frac{1}{L_N}g_N^A\right)-\left(x_N - \frac{1}{L_N}g_N\right)\right\rangle\\
    &=f_N + \frac{1}{L_N}\langle g_N, g_N - g_N^A\rangle\\
    &\geq f_N = f_\star + \tfrac{1}{2\tau_{N}}\|x_0-y_\star\|^2
\end{align*}
where the first considers $i=N$ in the definition of $f_\cH$, the second and third apply definitions and simplify, and the fourth uses that $\langle g_N, g_N^A\rangle \leq \|g_N\|^2$, and the final equality uses that $H_N=0$ due to Lemma~\ref{lem:tight_ineqs_SPPPA}. Since $\Psi_n = 1/\tau_{n,N} = \frac{1}{\tau_N}$, the claimed lower bound holds.
\end{proof}

\section{Practical Considerations and Numerics} \label{sec:pragmatics}

Both SPPPA and KLM admit natural limited-memory variants, storing only information about a bundle (subset) of past indices. For KLM, one can modify the subproblem~\eqref{eq:Kelley-solve} to only enforce a subset of the constraints. For SPPPA, one can modify the subproblem to construct $H'$ only from a subset of the $H_i$ and $Q_{\star,j}$. An inductive hypothesis built from a subset of past iterations remains valid, so limited memory SPPPA methods still provide convergence bounds of the form in Theorem~\ref{thm:SPPPA-rate}. For either method, one may hope that recent subgradients (or those selected by some other rule) are sufficient for most of the performance gains. Indeed, limited-memory variants of SPGM and its adaptive version have been observed to perform similarly to their full-memory versions; see~\cite{SPGM} and~\cite{ASPGM}.

In our preliminary numerical investigations of these methods below, we consider three variants of KLM and SPPPA: using a full memory (exactly as defined theoretically), using only the last ten iterations of information, and using only cutting planes active in the previous iteration's subproblem. (By active, we mean that the corresponding constraint is tight for KLM or the corresponding multiplier is strictly positive for SPPPA.)

\subsection{Dynamic Proximal Parameters and Inexact Proximal Operators}

The analysis of SPPPA can be extended almost verbatim to handle dynamic proximal parameters and inexact proximal operators.
% additionally handle dynamic proximal parameters and inexact.

First, although the class of algorithms we consider receive access to the full sequence of proximal parameters $L_n$ in advance, the subgame perfect algorithm SPPPA only uses $L_n$ in the $n$th iteration. In particular, 
% our theory presumes that the sequence of proximal parameters $L_n$ is set in advance, the algorithm does not need this. The parameter $L_n$ is only used in the $n$th iteration. So 
one can adopt \emph{any} adaptive rule for selecting this proximal parameter (a smaller proximal parameter
%  Particular policies for this choice would have to weigh the tradeoff that smaller $L_n$ allows 
allows increasing $\tau_n$ more rapidly between iterations at the expense of additional computational complexity within each iteration). For any sequence of proximal parameters chosen at runtime, one can conclude, in hindsight, that SPPPA is subgame perfect among the full class of algorithms that received the sequence of parameters ahead of time.

Second, the upper bounds can be extended to the setting where proximal operators are computed inexactly. Such approximate proximal point methods are a well-studied practical variant (see~\cite{Rockafellar1976_monotone,He2012,Barre2023principle}).
% Our analysis extends directly to settings where proximal subproblems are only solved within an $\epsilon_n\geq 0$-tolerance.
The analysis in the exact setting uses the fact that for the proximal point $y_n$, the vector $g_n = L_n(x_n-y_n)$ satisfies
$$ f(x) - f(y_n) - \langle g_n, x-y_n\rangle \geq 0 \qquad \forall x\in\mathbb{R}^d.$$
That is, $g_n\in \partial f(y_n)$. We then concluded the nonnegativity of $H_n$ by the recurrence:
$$ H_n = \sum_{i} \mu_i H_i + \sum_{i} \lambda_{\star,i}Q_{\star,i} + (\tau_n - \tau')Q_{\star,n} + \tau' Q_{m,n}. $$

In the inexact setting, suppose at iteration $n$ we produce $y_n$ so that $g_n = L_n(x_n-y_n)$ satisfies
$$ f(x) - f(y_n) - \langle g_n, x-y_n\rangle \geq -\epsilon_n \qquad \forall x\in\mathbb{R}^d$$
for some error parameter $\epsilon_n \geq 0$.
% SPPPA's exact analysis was based entirely on the condition that $Q_{i,j}\geq 0$, corresponding to $g_j\in\partial f(y_j)$. This nonnegativity enabled the main proof recurrence 
Such an inexact solve ensures that $Q_{i,n}\geq -\epsilon_n$ for all $i$. This allows us to track the violation on nonnegativity of $H_n$, i.e., a quantity $\delta_n\geq 0$ so that $H_n \geq - \delta_n$. This bound can be updated according to the above recursion
\begin{equation}\label{eq:inexact-accumulation}
	\delta_n = \sum_{i} \mu_i \delta_i + \sum_{i} \lambda_{\star,i}\epsilon_i + \tau_n\epsilon_n.
\end{equation}
Note, if $\epsilon_n$ is controlled by a user setting subproblem tolerances, this accumulated error $\delta_n$ can be tracked explicitly at runtime, possibly providing feedback on how to select future tolerances and proximal parameters.
Ultimately, when running SPPPA with inexact subproblem solutions, our final guarantees remain nearly the same: Rearranging $H_N\geq -\delta_N$ yields
$$ f_N - f_\star \leq \frac{\frac{1}{2}\|x_0-x_\star\|^2 + \delta_N}{\tau_N}. $$

\subsection{Numerical Demonstrations}

We compare the performance of KLM and SPPPA (both exact and inexact) with that of minimax optimal fixed-step methods across three numerical settings. For each of the dynamic algorithms, we consider both full-memory and limited-memory variants. In each setting, we fix the dimension $d=200$, the initialization $x_0=0$, and measure convergence via the suboptimality at each iteration. Proximal parameters in exact and inexact SPPPA are fixed as $L=5$. We plot performance as a function of oracle calls.
% Future work optimizing numerical implementations and subroutine costs may be able to determine wall-clock efficiency.
% Here we only aim to examine the improvement between schemes given equal information query budgets.
Our implementations can be found here\footnote{See~\url{https://github.com/bgrimmer/NonsmoothSubgamePerfectMethods}}.

For the setting of subgradient oracles, we consider a synthetic hard instance in dimension $d$ with low-dimensional structure (parameterized by $r<d$) around the optimal solution. Given values $a_i$, signs $\sigma_j\in\{\pm1\}$, and $M>1$, we define $x_\star = \sum_{j=1}^r \sigma_j a_j e_j $ and
\[
f(x)
=
\max\left\{
\max_{1\le j\le r} |x_j-(x_\star)_j|,
\; M \langle u, x\rangle - M(\|x_\star\|+5)
\right\}
\]
for some unit $u$ orthogonal to the first $r$ basis directions. Note that $f$ is $M$-Lipschitz with minimizer $x_\star$ and $f_\star=0$. At this minimizer, the final, steepest facet of $f$ is not active.
For our experiments, we set $r=10$, $a_i=i+2$, and $M=20$. The direction $u$ is sampled uniformly at random from the unit sphere on its last $d-r$ coordinates and $\sigma_j$ independently and uniformly from $\{\pm1\}$.

For the setting of exact proximal operators, we consider minimizing a structured quadratic problem. For orthogonal
$U\in\mathbb{R}^{m\times r}$ and $V\in\mathbb{R}^{d\times r}$ and values
$0<\sigma_1\le\cdots\le\sigma_r$, set
\[
A = U \operatorname{diag}(\sigma_1,\ldots,\sigma_r) V^\top,
\qquad
b = A x_\star,
\qquad
x_\star\in {\rm range}(V).
\]
Define $f(x)=\frac12\|Ax-b\|^2$ and note that evaluating the proximal operator corresponds to solving a least squares problem with improved conditioning. For our numerics, we set $m=120$ and $r=8$ with $\sigma_1,\dots,\sigma_8$ uniformly spaced in $[1,10^2]$ in logscale. We generate $U$ and $V$ by QR factorizing Gaussian random matrices and set $x_\star$ as $Vc$ with Gaussian random $c\in\mathbb{R}^8$.

For the setting of inexact proximal operators, we consider a LASSO problem, defined by 
\[
f(x)
=
\frac12 (x-c)^\top H(x-c)+\lambda\|x\|_1 .
\]
for some $H = V\operatorname{diag}(h_1,\ldots,h_d)V^\top$ with $0<h_1\leq\cdots\leq h_d$ and $V$ orthogonal.
We fix a sparse minimizer $x_\star$ with 20 nonzero entries by setting $c=x_\star+H^{-1}q$ for some $q\in\lambda\partial\|x_\star\|_1$. Each proximal oracle call approximately solves
\[
\min_y\left\{
f(y)+\frac{L}{2}\|y-x_n\|^2
\right\}
\]
using 80 iterations of a linearly convergent FISTA method. Note that the inexactness $\epsilon_n$ of such a calculation can also be computed as $\epsilon_n = \max_x\left\{f(y_n)+\langle g_n,x-y_n\rangle-f(x)\right\}$. To track error accumulation, our numerics additionally compute this to high precision using one thousand iterations of FISTA. We generate problem data by sampling a random orthogonal matrix $V$, spacing $h_i$ uniformly in logscale in $[10^{-2},10^{3}]$, and setting $\lambda=0.02$.

Figure~\ref{fig:numerics} compares with the standard minimax optimal subgradient method and OPPA on each of these problems with $N=80$. Each subgame perfect method is run with full memory, a fixed recent memory, or storage of previously active facets.

KLM converges rapidly once the active collection of facets at the minimizer is observed. This closely matches the numerics reported by~\cite{drori2016Kelley}. The limited memory variants here are much less performant, either failing to make rapid progress or having progress deteriorate every time the steep facet leaves the bundle. This points towards an opportunity for future work to better aggregate information from past cuts, an idea prevalent in the bundle method literature.

Across our experiments, SPPPA (both with full and limited memory) improves on OPPA by orders of magnitude. In the inexact setting, this improvement is substantial as errors~\eqref{eq:inexact-accumulation} accumulate more slowly. This behavior in the inexact setting potentially warrants future exploration.

\begin{figure}
	\includegraphics[width=\linewidth]{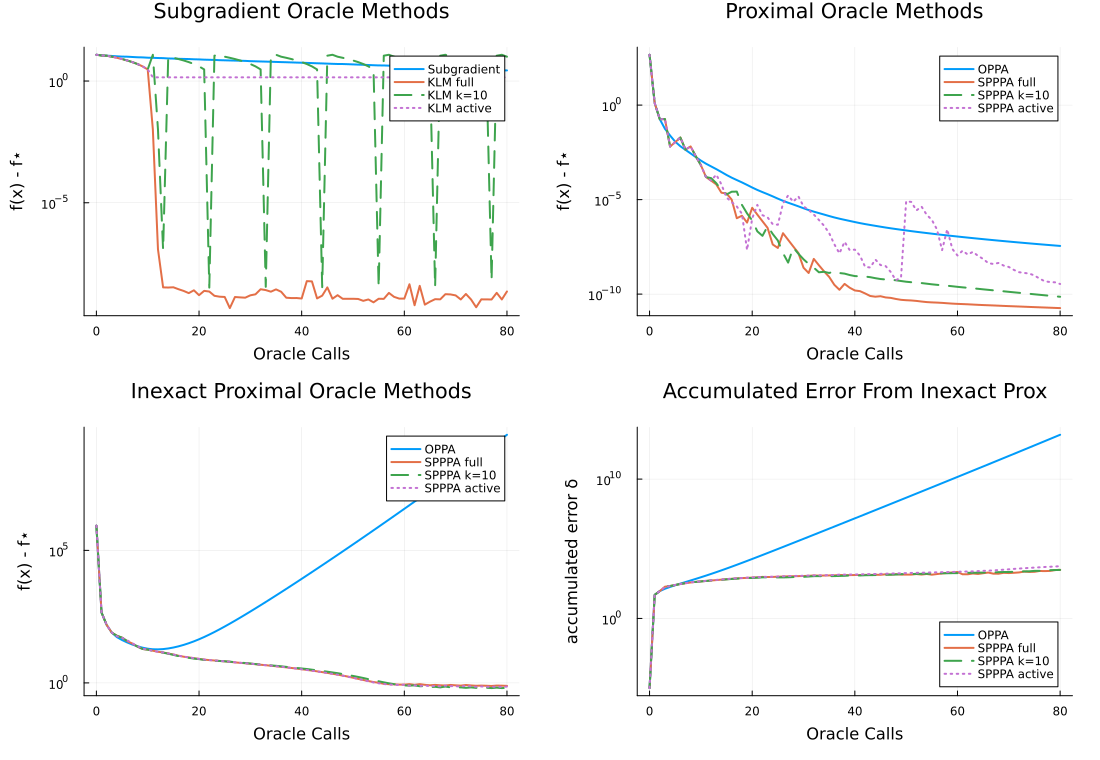}
	\caption{Preliminary demonstrations of subgame perfect methods on synthetic problem instances. Full (and certain limited memory) implementations here outperform traditional minimax optimal methods.} \label{fig:numerics}
\end{figure}

\section{Conclusion} \label{sec:conclusion}
We established subgame perfect algorithms for the settings of subgradient methods and proximal point algorithms. In particular, for subgradient methods, the Kelley cutting plane-Like Method of~\cite{drori2016Kelley} is not only minimax optimal (as previously known), but also subgame perfect. For proximal point methods, we introduced a new extension of OPPA, which we call SPPPA, that dynamically optimizes its induction at every step and is subgame perfect. By constructing dynamic lower bounding instances as a function of observed first-order responses, we established that these methods are guaranteed to not only attain the minimax optimal convergence guarantee over the family of all convex problem instances, but over every subclass of problem instances restricted to agree with the first-order information seen up to any iteration $n$. 

This game-theoretic notion was first applied to gradient methods in smooth convex minimization by~\cite{SPGM} and extended to adaptive backtracking settings in~\cite{ASPGM}. Many more settings exist where minimax optimal methods have been developed that represent fruitful opportunities to develop subgame perfect methods enabling optimal adaptation. Optimal methods for general fixed point and monotone operator settings have been considered by~\cite{lieder2021convergence,kim2021AcceleratedPPM}, representing one opportunity. A minimax optimal proximal gradient method (OptISTA) for convex, composite optimization was recently developed by~\cite{OPTIsta}. A subgame perfect extension of OptISTA also represents an opportunity. Additionally, it is known that there are multiple minimax optimal subgradient methods (for example, the classic subgradient method and the subspace search elimination method of~\cite{drori2019efficient}). It would be of interest to determine if there exist other subgame perfect subgradient methods, distinct from KLM.

\paragraph{Acknowledgments.} This work was supported in part by the Air Force Office of Scientific Research under award number FA9550-23-1-0531. Benjamin Grimmer was additionally supported as a fellow of the Alfred P. Sloan Foundation.

{\small
\bibliographystyle{unsrt}
\bibliography{bibliography}
}

\appendix

\section{Deferred Proofs of Lemmas in SPPPA's Lower Bound}

\subsection{Proof of Lemma~\ref{lem:inner_product_formulas}}
We handle the $i=n$ and $i>n$ cases separately.
First, suppose $i=n$. Since $i>j$, necessarily $j<n$, so only the first case of the lemma statement is relevant.
We write
\[
\langle g_j, y_n - y_j\rangle
= \langle g_j, y_n - z'\rangle + \langle g_j, z' - y_j\rangle.
\]
By construction, $g_n$ is orthogonal to $\spann\{g_0,\dots,g_{n-1}\}$, and hence to $g_j$ for all $j<n$. Since $y_n = x_n - \tfrac{1}{L_n}g_n$, we have
\[
\langle g_j, y_n - z'\rangle
= \langle g_j, x_n - z'\rangle.
\]
From the definition of $x_n$ as $\frac{\tau'}{\tau_n}y_m + \frac{\tau_n-\tau'}{\tau_n}z'$, we then have that
\[
\langle g_j, x_n - z'\rangle
= \frac{\tau'}{\tau_n}\,\langle g_j, y_m - z'\rangle.
\]
Combining the displays gives the claim in this first case of $i=n$.

For the remainder of the proof, assume $i>n$. We expand
\begin{equation}\label{eq:gj_yi_yj_split}
	\langle g_j, y_i - y_j\rangle
	= \langle g_j, y_i - z_i\rangle
	+ \langle g_j, z_i - y_j\rangle.
\end{equation}
By definition, $z_i=z'- \sum_{\ell=n}^{i-1}(\tau_{\ell}-\tau_{\ell-1})g_\ell$. Thus the second term in this expansion is
\begin{equation} \label{eq:second-expansion-term}
	\langle g_j, z_i - y_j\rangle = \begin{cases}
		\langle g_j, z'-y_j\rangle & \text{if } j <n\\
		f_\star-f_{j} & \text{if } j\geq n
	\end{cases}
\end{equation}
where the $j\geq n$ case follows from the following calculation
\begin{align*}
	\langle g_j, z_i - y_j\rangle &= \langle g_j, z_{j+1} - y_j\rangle \\
	&= -\left((\tau_j - \tau_{j-1}) - \frac{1}{L_j} \right)\|g_j\|^2\\
	&= -\left(f_{j-1} - f_\star - \frac{1}{L_j}\|g_j\|^2\right)\\
	&= f_\star - f_j
\end{align*}
where the first two lines use Lemma~\ref{lem:orthogonality}, noting the orthogonality of $g_j$ to future subgradients $g_\ell$ and that $\langle g_j, z_j - x_j\rangle=0$ after plugging in definitions of $z_{j+1}$ and $y_j$, the third applies $\|g_j\|^2 = \frac{f_{j-1}-f_\star}{\tau_{j}-\tau_{j-1}}$ and the fourth applies $f_j = f_{j-1} - \frac{1}{L_j}\|g_j\|^2$.

For the first term in this expansion, we use the inductive definitions of $x_i,y_i,z_{i+1}$ and $\tau_{i}$ to write
\begin{align*}
	 \langle g_j, y_i - z_i\rangle &=  \langle g_j, x_i - z_i\rangle\\
	 &=\frac{\tau_{i-1}}{\tau_i}\langle g_j, y_{i-1} - z_i\rangle\\
	 &=\begin{cases}
	 	\frac{\tau_n}{\tau_{n+1}}\langle g_j, y_n - z' + (\tau_n-\tau')g_{i-1}\rangle & \text{if } i=n+1\\
	 	\frac{\tau_{i-1}}{\tau_{i}}\langle g_j, y_{i-1} - z_{i-1} + (\tau_{i-1}-\tau_{i-2})g_{i-1}\rangle & \text{if } i\geq n+2
	 \end{cases}
\end{align*}
Observing that if $j<i-1$, we have $\langle g_j,g_{i-1}\rangle=0$, we unroll this recurrence as follows. The recurrence continues unchanged until an index where the added direction is not orthogonal, i.e., $g_{i-1}$ equals $g_j$. For $j<n$, this is the initial step. For $j=n$, this is the step from $n+1$ to $n$. For $j>n$, this is the step from $j+1$ to $j$. Hence,
\begin{align*}
	\langle g_j, y_i - z_i\rangle &=\begin{cases}
		\frac{\tau'}{\tau_{i}}\langle g_j, y_m - z'\rangle & \text{if } j<n,\\[0.25em]
		\frac{\tau_n}{\tau_{i}}\langle g_n, y_n - z' + (\tau_{n}-\tau')g_{n}\rangle & \text{if } j=n,\\[0.25em]
		\frac{\tau_j}{\tau_{i}}\langle g_j, y_j - z_j + (\tau_{j}-\tau_{j-1})g_{j}\rangle & \text{if } j>n.
	\end{cases}
\end{align*}
The latter two cases can be further simplified by Lemma~\ref{lem:orthogonality} since $\langle g_n,x_n-z'\rangle=0$. Hence, using $y_n=x_n-\tfrac{1}{L_n}g_n$ and  $\|g_n\|^2 = \frac{f_{m}-f_\star}{\tau_{n}-\tau'}$ gives
\begin{align*}
	\langle g_n, y_n-z'+(\tau_n-\tau')g_n\rangle =\left(\tau_n-\tau'-\frac{1}{L_n}\right)\|g_n\|^2 =(f_{m}-f_\star)-\frac{1}{L_n}\|g_n\|^2
	= f_n-f_\star.
\end{align*}
Similarly, for $j>n$, by Lemma~\ref{lem:orthogonality}, $\langle g_j,x_j-z_j\rangle=0$. Hence, using $y_j=x_j-\tfrac{1}{L_j}g_j$ and $\|g_j\|^2 = \frac{f_{j-1}-f_\star}{\tau_{j}-\tau_{j-1}}$ gives
\begin{align*}
	\langle g_j, y_j-z_j+(\tau_j-\tau_{j-1})g_j\rangle =\left(\tau_j-\tau_{j-1}-\frac{1}{L_j}\right)\|g_j\|^2
	=(f_{j-1}-f_\star)-\frac{1}{L_j}\|g_j\|^2
	= f_j-f_\star.
\end{align*}

Applying these simplifications, we find
\begin{equation}\label{eq:first-expansion-term}
	\langle g_j, y_i - z_i\rangle =\begin{cases}
		\frac{\tau'}{\tau_{i}}\langle g_j, y_m - z'\rangle & \text{if } j<n,\\[0.25em]
		\frac{\tau_j}{\tau_{i}}(f_j - f_\star) & \text{if } j\geq n.
	\end{cases}
\end{equation}
Combining~\eqref{eq:second-expansion-term} and~\eqref{eq:first-expansion-term} yields our claimed formula when $i>n$, completing our proof.

\subsection{Proof of Lemma~\ref{lem:SPPPA-monotone}}
Define $A_i = \tau_i(f_i-f_\star)$ as our main quantity of interest in this lemma.
For convenience and as a minor abuse of notation, we introduce an index $i=n-1$ with
\[
\tau_{n-1} \coloneqq \tau', \qquad f_{n-1} \coloneqq f_m, \qquad
A_{n-1} \coloneqq \tau'(f_m - f_\star).
\]
With this convention, our construction of future iterate values can be written uniformly as follows: for every $i\in\{n,\dots,N\}$,
\begin{align*}
	\delta_i &= \tau_i - \tau_{i-1}
	= \frac{1}{L_i}\left(1+\sqrt{1+2L_i\tau_{i-1}}\right),\\
	\|g_i\|^2 &= \frac{f_{i-1} - f_\star}{\delta_i}, \qquad f_i = f_{i-1} - \frac{1}{L_i}\|g_i\|^2.
\end{align*}
For any $i\ge n$, we have
$f_i - f_\star
= f_{i-1} - \frac{1}{L_i}\|g_i\|^2 - f_\star
= (f_{i-1} - f_\star)\left(1 - \frac{1}{L_i\delta_i}\right),$
and so
\begin{align*}
	A_i - A_{i-1}
	&= \tau_i(f_i - f_\star) - \tau_{i-1}(f_{i-1} - f_\star)\\
	&= (f_{i-1} - f_\star)\left(\tau_i\left(1 - \frac{1}{L_i\delta_i}\right) - \tau_{i-1}\right).
\end{align*}
By construction, $\delta_i>0$ and $\|g_i\|^2\ge 0$, so $f_{i-1}-f_\star = \|g_i\|^2\delta_i\ge 0$. Hence it suffices to show
\begin{equation}\label{eq:Ai_step_bracket}
	\tau_i\left(1 - \frac{1}{L_i\delta_i}\right) - \tau_{i-1} \;\ge\; 0
	\qquad\text{for all }i\in\{n,\dots,N\}.
\end{equation}
Using $\delta_i = \tau_i - \tau_{i-1}$, we rewrite
\begin{align*}
	\tau_i\left(1 - \frac{1}{L_i\delta_i}\right) - \tau_{i-1}
	= \left(\tau_i - \tau_{i-1}\right) - \frac{\tau_i}{L_i\delta_i}
	= \delta_i - \frac{\tau_i}{L_i\delta_i}
	= \frac{L_i\delta_i^2 - \tau_i}{L_i\delta_i}.
\end{align*}
Since $\delta_i,L_i>0$, it remains to verify that $L_i\delta_i^2 \ge \tau_i.$
By definition, $\delta_i$ is a root of the quadratic equation
\begin{equation*}
    \frac{L_i\delta_i^2}{2} - \delta_i - \tau_{i-1} = 0.
\end{equation*}
Recognizing, $\tau_{i-1} + \delta_i = \tau_i$, we deduce that $L_i\delta_i^2 \geq \frac{L_i\delta_i^2}{2} = \tau_i$, completing the proof.

% By definition of $\delta_i$, we have that
% \begin{align*}
% 	\delta_i^2
% 	&= \frac{1}{L_i^2}\left(1+\sqrt{1+2L_i\tau_{i-1}}\right)^2\\
% 	&= \frac{1}{L_i^2}\left(1 + 2\sqrt{1+2L_i\tau_{i-1}} + (1+2L_i\tau_{i-1})\right)\\
% 	&= \frac{2}{L_i^2}\left(L_i\tau_{i-1} + \sqrt{1+2L_i\tau_{i-1}} + 1\right).
% \end{align*}
% Observing that $\frac{\tau_i}{L_i}
% = \frac{\tau_{i-1} + \delta_i}{L_i}
% = \frac{\tau_{i-1}}{L_i} + \frac{1}{L_i^2}\left(1+\sqrt{1+2L_i\tau_{i-1}}\right)$, we have
% \begin{align*}
% 	\delta_i^2 - \frac{\tau_i}{L_i}
% 	&= \frac{1}{L_i^2}\left(2\left(L_i\tau_{i-1} + \sqrt{1+2L_i\tau_{i-1}} + 1\right)
% 	- \tau_{i-1}L_i - \left(1+\sqrt{1+2L_i\tau_{i-1}}\right)\right)\\
% 	&= \frac{1}{L_i^2}\left(L_i\tau_{i-1} + \sqrt{1+2L_i\tau_{i-1}} + 1\right)\geq 0,
% \end{align*}
% since $L_i>0$ and $\tau_{i-1}\ge 0$. This verifies \eqref{eq:Ai_step_bracket}, establishing the claimed monotonicity $A_{i}\geq A_{i-1}$.

\subsection{Proof of Lemma~\ref{lem:tight_ineqs_SPPPA}}
First, $Q_{i-1,i}=0$ for all $i>n$:
\begin{align*}
	f_{i} + \langle g_i, y_{i-1} - y_i\rangle &= f_{i} + \left\langle g_i, y_{i-1} - \left( \frac{\tau_{i-1}}{\tau_i}y_{i-1}+\frac{\tau_{i}-\tau_{i-1}}{\tau_i}z_i - \frac{1}{L_i}g_i\right) \right\rangle
	= f_i + \frac{1}{L_i}\|g_i\|^2 = f_{i-1}
\end{align*}
where the second equality uses that $g_i$ is orthogonal to $\mathrm{span}\{g_0,\dots,g_{i-1}\}$ which contains both $y_{i-1}-x_0$ and $z_i-x_0$. Second, $Q_{\star, i}=0$ for all $i>n$, as
\begin{align*}
	f_i + \langle g_i, y_\star - y_i\rangle & = f_{i-1} + \langle g_i, z_{N+1} - y_{i-1}\rangle = f_{i-1} - (\tau_{i}-\tau_{i-1})\|g_i\|^2 = f_\star
\end{align*}
where the first equality uses that $Q_{i-1,i}=0$ and substitutes the definition $y_\star=z_{N+1}$, the second uses the definition of $z_{N+1}$ and orthogonality of each $e_i$, and the third uses the fact that $\|g_i\|^2 = \frac{f_{i-1}-f_\star}{\tau_{i} - \tau_{i-1}}$.
For both arguments above, the case of $i=n$ is identical, replacing $\tau_{i-1}$, $z_i$, $y_{i-1}$, and $f_{i-1}$ with $\tau'$, $z'$, $y_m$, and $f_m$ respectively.

From strong duality (see Lemma~\ref{lem:dual_prox_full}), we know that the primal/dual optimizers
$(\mu,\lambda_\star)$ and $(\xi,w)$ with the associated
$z'=x_0+Z\mu-G\lambda_\star$ and $\frac{1}{\xi}=f_m-f_\star$ have
\begin{equation}\label{eq:phi_value_relation_full_proof}
	\tau' = \frac{\xi}{2}\|z'-x_0\|^2,\qquad \tau'(f_\star-f_m)=-\frac12\|z'-x_0\|^2.
\end{equation}
Evaluating $H'$ at $y_\star=z_{N+1}$, using \eqref{eq:phi_value_relation_full_proof}, and the Pythagorean theorem with the orthogonal directions $x_0-z'$ and $z'-y_\star$ yields
\[
H' = \tau'(f_\star-f_m)+\tfrac12\|x_0-y_\star\|^2 - \tfrac12\|z'-y_\star\|^2 = \tau'(f_\star-f_m)+\tfrac12\|x_0-z'\|^2  =  0.
\]
Finally, we verify that $H_i=0$ for all $i\geq n$. We do so inductively:
\begin{align*}
	H_n  &=  H' + (\tau_n-\tau') Q_{\star,n} + \tau' Q_{m,n}  =  0,\\
	H_{i} &= H_{i-1} + (\tau_i-\tau_{i-1}) Q_{\star,i} + \tau_{i-1} Q_{i-1,i}  = 0 \qquad \text{for\ } i>n
\end{align*}
where the equality on each line uses Lemma~\ref{lem:OPPA_induction} and the second equality on each line recognizes expressions that we have (inductively) shown to be zero.

\end{document}